\documentclass[12pt]{article}
\usepackage{amssymb,amsmath,amsthm,enumerate}

\newcommand{\Z}{{\mathbb Z}}
\newcommand{\X}{{\mathbb X}}

\newcommand{\A}{\mathcal{A}}
\newcommand{\I}{\mathcal{I}}

\newcommand{\OO}{\mathcal{O}}
\newcommand{\N}{\mathbb{N}}
\newcommand{\Q}{{\mathbb Q}}
\newcommand{\F}{{\mathbb F}}
\newcommand{\M}{\mathcal{M}}

\renewcommand{\a}{\mathfrak{a}}

\newcommand{\z}{\zeta}
\newcommand{\x}{\vec{x}}

\newcommand{\vomega}{\vec{\omega}}
\newcommand{\vOmega}{\vec{\Omega}}
\newcommand{\vepsilon}{\vec{\epsilon}}
\newcommand{\vE}{\vec{E}}
\newcommand{\onebar}{\overline{1}}
\newcommand{\Rbar}{\overline{R}}
\newcommand{\Jbar}{\overline{J}}

\newcommand{\gammabar}{\overline{\gamma}}
\newcommand{\alphabar}{\overline{\alpha}}

\newcommand{\deltabar}{\overline{\delta}}

\newcommand{\Gbar}{\overline{G}}
\newcommand{\Lambdabar}{\overline{\Lambda}}
\newcommand{\ra}{\rightarrow}
\newcommand{\lra}{\longrightarrow}
\newcommand{\dst}{\displaystyle}

\newcommand{\GL}{\rm{GL}}
\DeclareMathOperator{\ch}{char}
\DeclareMathOperator{\Tr}{Tr}
\DeclareMathOperator{\Span}{Span}
\DeclareMathOperator{\Gal}{Gal}
\DeclareMathOperator{\Ann}{Ann}

\newtheorem{theorem}{Theorem}
\newtheorem{lemma}[theorem]{Lemma}
\newtheorem{prop}[theorem]{Proposition}
\newtheorem{cor}[theorem]{Corollary}

\theoremstyle{definition}
\newtheorem{remark}[theorem]{Remark}
\newtheorem{definition}[theorem]{Definition}

\numberwithin{equation}{section}
\numberwithin{theorem}{section}

\title{Refined ramification breaks in characteristic
$p$}
\author{G. Griffith Elder \\
Department of Mathematics \\
University of Nebraska Omaha \\
Omaha, NE 68182 \\
USA \\[.2cm]
{\tt elder@unomaha.edu}
\and
Kevin Keating \\
Department of Mathematics \\
University of Florida \\
Gainesville, FL 32611 \\
USA \\[.2cm]
{\tt keating@ufl.edu}}

\begin{document}

\maketitle

\begin{abstract}
Let $K$ be a local field of characteristic $p$ and let
$L/K$ be a totally ramified elementary abelian
$p$-extension with a single ramification break $b$.
Byott and Elder defined the refined ramification breaks
of $L/K$, an extension of the usual ramification data.
In this paper we give an alternative definition for the
refined ramification breaks, and we use Artin-Schreier theory to
compute both versions of the breaks in some special
cases.
\end{abstract}

\section{Introduction}

Let $K$ be a local field whose residue field $\F$ is a perfect field
of characteristic $p$ and let $L/K$ be a finite totally ramified
Galois extension.  Let $G=\Gal(L/K)$ and set $[L:K]=d=ap^n$, with
$p\nmid a$.  Then the extension $L/K$ has at most $n$ positive lower
ramification breaks.  In certain cases (for instance, if $G$ is
cyclic) $L/K$ must have exactly $n$ positive ramification breaks.
When $L/K$ has fewer than $n$ positive ramification breaks one might
hope to replace the missing breaks with some other information.

One
attempt to supply the missing information was made by Fried
\cite{fried} and Heiermann \cite{heier}, who defined a set of data
which Heiermann called the ``indices of inseparability'' of $L/K$.
The indices of inseparability are equivalent to the usual ramification
data in the case where $L/K$ has $n$ positive ramification breaks, and
provide new information when $L/K$ has fewer than $n$ positive breaks.

Now consider the extreme situation where $L/K$ has a
single ramification break $b$, with $b>0$.  Then $G\cong
C_p^n$ for some $n\ge1$, where $C_p$ denotes the cyclic group of order
$p$ \cite[III,\,Th.\,4.2]{FV}.  In this setting Byott and Elder
\cite{new,necessity} defined ``refined ramification breaks'' for $L/K$
in terms of the action of $G$ on $L$:  If $\ch(K)=p$ set
$R=\F[G]$, while
if $\ch(K)=0$ let $R=W[G]$, where $W$ is the ring of Witt vectors over $\F$.  In
either case let $A$ denote the augmentation ideal of $R$.  Using
``truncated exponentiation'', the group $(1+A)/(1+A^p)$ can be given
the structure of a vector space over the residue field $\F$.  The
image of $G$ in this group spans an $\F$-vector space $\Gbar^{[\F]}$
of dimension $n$.  By considering the action of (coset representatives
of) elements of $\Gbar^{[\F]}$ on elements $\rho\in L$ one can define
new ramification breaks for $L/K$.  An early observation was that $n$
refined ramification breaks are produced if $\rho$ generates a normal
basis for $L/K$ \cite[Theorem 3.3]{new}, although it was unknown
how the values of these breaks depend upon the particular normal
basis generator chosen.

In \cite{necessity} Byott and Elder focused on the case where
$\ch(K)=0$, $K$ contains a primitive $p$th root of unity, and $G\cong
C_p^n$ with $n=2$.  In \cite{proc-lms}, it had
been observed that elements whose valuation is congruent to $b$ modulo
$p^n$ satisfy a ``valuation criterion'': any $\rho\in L$
such that $v_L(\rho)\equiv b\pmod{p^n}$ is
a normal basis generator for $L/K$.  For this reason, the refined
ramification breaks in \cite{necessity} were defined in terms of the
action of $\Gbar^{[\F]}$ on valuation criterion elements of $L$.
Byott and Elder used Kummer theory to calculate the values of the two
refined ramification breaks, and showed that these values are
independent of choice of valuation criterion element.  They also
showed that in certain cases these new breaks give information about
the Galois module structure of $L$.  It remains an open question
whether the values of the refined ramification breaks are independent
of the choice of valuation criterion element for totally
ramified $C_p^n$-extensions with a single ramification
break when $n\ge3$.
     
In this paper we once again consider totally ramified
$C_p^n$-extensions $L/K$ with a single ramification break $b>0$.  We
propose a new definition for the refined ramification breaks of $L/K$
which depends on the action of $\Gbar^{[\F]}$ on all of $L$, rather
than just on the valuation criterion elements.  This definition has
the advantage of being independent of all choices, and gives breaks
which are ``necessary'' for Galois module structure, as in
\cite{necessity} (see \S\ref{concluding}).  It has the
disadvantage that it is not obvious that it produces $n$ distinct
breaks.  We apply these definitions to a certain class of elementary
abelian $p$-extensions in characteristic $p$.  This class includes all
$C_p^2$-extensions with a single ramification break as well as the
``one-dimensional'' extensions from \cite{scaf} with a single
ramification break.  For the extensions in this class, we use the
results of \cite{large} to show that the two definitions of refined
ramification breaks give the same values, and then compute these
values in terms of Artin-Schreier equations.  In
Remark~\ref{semistable} another sufficient condition, due to Bondarko \cite{bon2}, is given for the
two definitions of refined ramification break to be equivalent.  We do
not know whether the two definitions for refined ramification breaks
are equivalent more generally.

The authors thank Nigel Byott for his careful reading of the paper,
and for asking about the statement that has become
Proposition~\ref{scafr}.

\section{Refined ramification breaks} \label{ref-brk-defn}

Let $K$ be a local field of characteristic $p$ with perfect residue
field $\F$, and let $L/K$ be a totally ramified $C_p^n$-extension with
a single ramification break $b>0$.  In this section we give two
definitions for the refined ramification breaks (or refined breaks) of
$L/K$.  Our definition of VC-refined breaks (where VC
stands for ``valuation criterion'') is essentially the
same as the definition of refined breaks given in \cite{necessity}. As
mentioned in the introduction, for any valuation
criterion element $\rho$
this definition is guaranteed to produce $n$ distinct refined breaks,
but we do not know that the values of the refined breaks are
independent of the choice of $\rho$.  Our definition of
SS-refined breaks (where SS stands for ``smallest
shift'') differs from the definition in \cite{necessity}
in that it depends on the action of $\F[G]$ on all the elements of
$L$, not just the action on valuation criterion elements. The values
of the refined breaks produced by this definition are independent of
all choices, but it is an open question whether the definition always
produces a full complement of $n$ refined breaks.  Each definition
comes in different versions which depend on a parameter $k$ satisfying
$2\le k\le p$.  The VC$_k$-refined breaks and the SS$_k$-refined
breaks are defined using cosets of $A^k$, where $A$ is the
augmentation ideal of $R=\F[G]$.  These various definitions are not
obviously equivalent, but in Corollary~\ref{same} we
give sufficient conditions for the set of VC$_k$-refined
breaks to be equal to the set of SS$_k$-refined breaks,
and in Theorem~\ref{main} we give stronger conditions
under which these sets are independent of $k$.  It would
certainly be useful to have a better understanding of
when our various sets of refined breaks are the same and
when they differ.  In all the examples we are able to
compute, the sets of VC$_k$- and SS$_k$-refined breaks
are equal and independent of $k$.  Thus it would be
interesting  to find an example of an extension
$L/K$ for which, say, the VC$_k$-refined breaks are
different from the VC$_{k'}$-refined breaks for some
$2\le k<k'\le p$.

Since the residue field $\F$ of $K$ is perfect, we have
$K\cong\F((t))$.  Let $v_L:L\ra\Z\cup\{\infty\}$ be the normalized
valuation on $L$.  Then $\OO_L=\{x\in L:v_L(x)\ge0\}$ is the ring of
integers of $L$ and $\M_L=\{x\in L:v_L(x)\ge1\}$ is the maximal ideal
in $\OO_L$.  Since $G=\Gal(L/K)$ is an elementary abelian $p$-group of
rank $n$, $L$ is the compositum of $n$ fields $L_1,\ldots,L_n$ which
are cyclic degree-$p$ extensions of $K$.  Hence for $1\le i\le n$
there is $\alpha_i\in K$ such that $L_i$ is the splitting field of the
Artin-Schreier polynomial $X^p-X-\alpha_i$.  Since $b$ is the unique
ramification break of $L_i/K$, we may assume that $v_K(\alpha_i)=-b$.
Let $\beta\in K$ satisfy $v_K(\beta)=-b$.  Then there are
$\omega_1,\ldots,\omega_n\in\F$ and $\epsilon_1,\ldots,\epsilon_n\in
K$ such that $v_K(\epsilon_i)>-b$ and
$\alpha_i=\omega_i^{p^n}\beta+\epsilon_i$ for $1\le i\le n$.
Furthermore, since $b$ is the only ramification break of
$L/K$, the coefficients
$\omega_1,\ldots,\omega_n\in\F$ must be linearly independent over
$\F_p$.

     Let $\wp:K\ra K$ be the Artin-Schreier map,
defined by $\wp(x)=x^p-x$.  By replacing $\epsilon_i$
with $\epsilon_i'\in\epsilon_i+\wp(K)$ we may assume
that either $v_L(\epsilon_i)<0$ and
$p\nmid v_L(\epsilon_i)$, or $\epsilon_i\in\F$.  Set
$e_i=-v_L(\epsilon_i)$; if $\epsilon_i\not=0$ then $0\le e_i<b$.
Also define
\[\vomega=\begin{bmatrix}\omega_1\\ \vdots\\ \omega_n
\end{bmatrix}
\hspace{1cm}
\vepsilon=\begin{bmatrix}\epsilon_1\\ \vdots\\
\epsilon_n
\end{bmatrix}.\]
We say that $(\beta,\vomega,\vepsilon)$ is
Artin-Schreier data for the $C_p^n$-extension $L/K$.
Of course, $(\beta,\vomega,\vepsilon)$ is not uniquely
determined by $L/K$, but $(\beta,\vomega,\vepsilon)$
does determine $L$ as an extension of $K$.  By choosing
$\beta=\alpha_1$ we may assume that $\omega_1=1$ and
$\epsilon_1=0$.

Define the truncated exponential and truncated logarithm polynomials
by
\begin{align*}
e_p(X)&=\sum_{i=0}^{p-1}\,\frac{1}{i!}X^i \\
\ell_p(X)&=\sum_{i=0}^{p-1}\,\frac{(-1)^i}{i!}(X-1)^i.
\end{align*}
Note that $e_p(X)$ is not the same as the ``truncated
exponentiation'' used in \cite{new,necessity,large}.
Since the congruences
\begin{alignat}{2}
\ell_p(e_p(X))&\equiv X&&\pmod{X^p} \nonumber \\
e_p(\ell_p(1+X))&\equiv1+X&&\pmod{X^p} \nonumber \\
e_p(X+Y)&\equiv e_p(X)e_p(Y)&&\pmod{(X,Y)^p}
\nonumber \\
\ell_p((1+X)(1+Y))&\equiv\ell_p(1+X)+\ell_p(1+Y)
&&\pmod{(X,Y)^p} \label{lpXY}
\end{alignat}
are valid in $\Q[X,Y]$, and involve polynomials
with coefficients in $\Z_{(p)}$, they are valid over
$\F_p\cong\Z_{(p)}/p\Z_{(p)}$,
and hence also over $\F$ and over $K$.

     For $\omega\in K$ and $1\le i<p$ define
\[\binom{\omega}{i}=
\frac{\omega(\omega-1)(\omega-2)\dots(\omega-(i-1))}{i!}
\in K.\]
Also define $\dst\binom{\omega}{0}=1$ and
$\dst\binom{\omega}{-1}=0$.  Let $\psi(X)\in XK[X]$.
Following \cite[1.1]{new} we define the truncated
$\omega$ power of $1+\psi(X)$ to be the polynomial
\[(1+\psi(X))^{[\omega]}
=\sum_{i=0}^{p-1}\binom{\omega}{i}\psi(X)^i\] obtained by truncating
the binomial series.  This is what was called ``truncated
exponentiation'' in \cite{new,necessity,large}. We have the following
(cf.\ \cite[Prop.\,2.2]{refined}).

\begin{prop} \label{lpXZ}
$\ell_p((1+X)^{[\omega]})\equiv
\omega\ell_p(1+X)\pmod{X^p}$.
\end{prop}

\begin{proof} Since the congruence $e_p(X)^{[Z]}
\equiv e_p(Z X)\pmod{X^p}$ holds in $\Q[X,Z]$,
and involves polynomials with coefficients in
$\Z_{(p)}$, it is valid over $\F$.  By replacing $X$
with $\ell_p(1+X)$ and $Z$ with $\omega\in K$, we get
\begin{alignat*}{2}
(1+X)^{[\omega]}&\equiv e_p(\omega\ell_p(1+X))&&\pmod{X^p}.
\end{alignat*}
Applying $\ell_p$ to this congruence gives the
proposition.
\end{proof}

     Recall that $R=\F[G]$ and that $A$ is the augmentation
ideal of $R$.  

\begin{cor} \label{log}
Let $m\ge1$, let $\alpha,\beta\in1+A^m$, and let
$\omega\in\F$.  Then
\begin{align*}
\ell_p(\alpha^{[\omega]})&=\omega\ell_p(\alpha) \\
\ell_p(\alpha\beta)&\equiv
\ell_p(\alpha)+\ell_p(\beta)\pmod{A^{pm}}.
\end{align*}
\end{cor}

\begin{proof} This follows from Proposition~\ref{lpXZ}
and congruence (\ref{lpXY}) by setting $X=\alpha-1$ and
$Y=\beta-1$.  The first formula is an equality rather
than a congruence because $(\alpha-1)^p=0$.
\end{proof}

     Fix $k$ such that $2\le k\le p$ and set
$\Rbar=R/A^k$.  For $\gamma\in R$ let
$\gammabar=\gamma+A^k$ be the image of $\gamma$ in
$\Rbar$.
Note that $1+A$
and $1+A^k$ are subgroups of $R^{\times}$, and that
$(1+A)/(1+A^k)$ is isomorphic to the image of $1+A$ in
$\Rbar^{\times}$.
For $\gamma\in1+A$ and $\omega\in\F$ let
$\gammabar^{[\omega]}=\gamma^{[\omega]}+A^k$ denote the
image of $\gamma^{[\omega]}$ in $\Rbar$.  
The function $\Lambda_p:1+A\ra A$
defined by $\Lambda_p(\alpha)=\ell_p(\alpha)$ induces a
bijection
\[\Lambdabar_p:(1+A)/(1+A^k)\lra A/A^k.\]
By Corollary~\ref{log} this map is a group isomorphism.
Furthermore, defining scalar multiplication by
$\omega\cdot\alphabar=\alphabar^{[\omega]}$ makes 
$(1+A)/(1+A^k)$ a vector space over
$\F$, and $\Lambdabar_p$ an isomorphism of
$\F$-vector spaces.  Let $\Gbar$ denote the image
of $G$ in $(1+A)/(1+A^k)$, and let
$\Gbar^{[\F]}$ be the $\F$-span of $\Gbar$.

     Let $\M_L^{d_{L/K}}$ be the different of the
extension $L/K$.  We say that $\rho\in L^{\times}$ is a
valuation criterion element for $L/K$ if
$v_L(\rho)\equiv-d_{L/K}-1\pmod{p^n}$.  In \cite{valp}
it is shown that every valuation criterion element
generates a normal basis for $L/K$.  Since the only
ramification break of $L/K$ is $b$ we have
$d_{L/K}=(p^n-1)(b+1)$.  Therefore the valuation
criterion for $\rho\in L$ is
$v_L(\rho)\equiv b\pmod{p^n}$, in agreement with
\cite{proc-lms}.  Let $\rho$ be a
valuation criterion element for $L/K$.  For
$\gammabar\in\Gbar^{[\F]}$ we define
\[i_{\rho}(\gammabar)=\max\{v_L((\gamma'-1)({\rho})):
\gamma'\in\gammabar\}.\]

\begin{definition} \label{VC}
The set of refined ramification breaks of $L/K$ with
respect to $\rho$ and $\Rbar=R/A^k$ is defined to be
\[B_{\rho,k}=\{i_{\rho}(\gammabar)-v_L(\rho):
\gammabar\in\Gbar^{[\F]},\;\gammabar\not=\onebar\}.\]
We say that the elements of $B_{\rho,k}$ are the
$(\rho,k)$-refined breaks of $L/K$.  If the set
$B_{\rho,k}$ is the same for all $\rho$ such that
$v_L(\rho)\equiv b\pmod{p^n}$ we define the set of
valuation criterion refined breaks of $L/K$ (with
respect to $\Rbar=R/A^k$) to be $B_{\rho,k}$ for any 
valuation criterion element $\rho$.  In this
case we say that the elements of $B_{\rho,k}$ are the
VC$_k$-refined breaks of $L/K$.
\end{definition}

 The argument used to prove Theorem~3.3 of \cite{new} shows that
$B_{\rho,k}$ consists of $n$ distinct elements.  It
follows from the proof of \cite[Lemma~3]{necessity} that when $n=2$,
the VC$_k$-refined breaks of $L/K$ are defined for
$2\le k\le p$.

     We wish to give an alternative definition for
refined ramification breaks of $L/K$ which takes into
account the effect of $\gamma\in R$ on the valuations of
all the elements of $L$.  Motivated by the definition of
the norm of a linear operator, and also by the
definitions of $\mathfrak{C}_i$ and $\mathfrak{A}_i$ in
\cite[p.\,36]{bon2}, we set
\[\hat{v}_L(\gamma)=\min\{v_L(\gamma(x))-v_L(x):
x\in L^{\times}\}.\]
Then $\hat{v}_L(\gamma)=\infty$ if and only if
$\gamma=0$.  Furthermore, for $\gamma,\delta\in R$ we
have
\begin{align*}
\hat{v}_L(\gamma\delta)
&\ge\hat{v}_L(\gamma)+\hat{v}_L(\delta) \\
\hat{v}_L(\gamma+\delta)
&\ge\min\{\hat{v}_L(\gamma),\hat{v}_L(\delta)\}.
\end{align*}
Therefore $\hat{v}_L$ is a pseudo-valuation on $R$ (see
\cite[p.\,108]{rees}).

\begin{lemma} \label{all}
Let $\gamma\in R$ and let $x\in L^{\times}$ satisfy
$\hat{v}_L(\gamma)=v_L(\gamma(x))-v_L(x)$.  Then for
every $y\in L^{\times}$ such that $v_L(y)\equiv v_L(x)
\pmod{p^n}$ we have $\hat{v}_L(\gamma)
=v_L(\gamma(y))-v_L(y)$.
\end{lemma}

\begin{proof}
The assumption on $y$ implies that there is $c\in K$
such that $v_L(cx)=v_L(y)$ and $v_L(y-cx)>v_L(y)$.  Set
$z=y-cx$.  Then
\[v_L(\gamma(cx))-v_L(cx)=v_L(\gamma(x))-v_L(x)
=\hat{v}_L(\gamma),\]
so we have
\[v_L(\gamma(z))\ge\hat{v}_L(\gamma)+v_L(z) 
>\hat{v}_L(\gamma)+v_L(cx)=v_L(\gamma(cx)).\]
It follows that
\begin{align*}
v_L(\gamma(y))-v_L(y)
&=v_L(\gamma(cx)+\gamma(z))-v_L(cx+z) \\
&=v_L(\gamma(cx))-v_L(cx) \\
&=\hat{v}_L(\gamma). \qedhere
\end{align*}
\end{proof}

     For $\gammabar\in\Rbar$ define
\[\hat{v}_L(\gammabar)
=\max\{\hat{v}_L(\gamma'):\gamma'\in\gammabar\}.\]
Suppose $\gamma'\in\gammabar$, $\delta'\in\deltabar$
satisfy $\hat{v}_L(\gamma')=\hat{v}_L(\gammabar)$,
$\hat{v}_L(\delta')=\hat{v}_L(\deltabar)$.  Then
\begin{align*}
\hat{v}_L(\gammabar+\deltabar)
&\ge\hat{v}_L(\gamma'+\delta')
\ge\min\{\hat{v}_L(\gamma'),\hat{v}_L(\delta')\}
=\min\{\hat{v}_L(\gammabar),\hat{v}_L(\deltabar)\} \\
\hat{v}_L(\gammabar\deltabar)
&\ge\hat{v}_L(\gamma'\delta')
\ge\hat{v}_L(\gamma') +\hat{v}_L(\delta')
=\hat{v}_L(\gammabar)+\hat{v}_L(\deltabar).
\end{align*}
Therefore $\hat{v}_L$ is a pseudo-valuation on $\Rbar$.  
For $h\ge0$ define
\[\Jbar_h=\{\gammabar\in\Rbar:\hat{v}_L(\gammabar)
\ge h\}.\]
Since $\hat{v}_L$ is a pseudo-valuation on $\Rbar$ we
see that $\Jbar_h$ is an ideal in $\Rbar$.  We clearly
have $\Jbar_0=\Rbar$, $\Jbar_{h+1}\subset\Jbar_h$ for
$h\ge0$, and $\Jbar_h=\{\overline{0}\}$ for sufficiently
large $h$.  For $h\ge0$ let
\begin{equation} \label{filter}
\Gbar^{[\F]}_h=\{\gammabar\in\Gbar^{[\F]}:
\gammabar-\overline{1}\in\Jbar_h\}.
\end{equation}

\begin{definition} \label{SS}
Say $h\in\N\cup\{0\}$ is a smallest-shift ramification break of
$L/K$ (with respect to $\Rbar=R/A^k$) if
$\Gbar^{[\F]}_{h+1}\not=\Gbar^{[\F]}_h$.  In this case
we say that $h$ is an SS$_k$-refined break of $L/K$.
\end{definition}

\begin{remark} \label{semistable}
Let $\rho$ be a valuation criterion element for $L/K$.
If the extension $L/K$ is ``semistable'' in the sense of
Definition~3.1.1 of \cite{bon2} then by Theorem~4.4 of
the same paper we have
$\hat{v}_L(\gamma)=v_L(\gamma(\rho))-v_L(\rho)$ for
every $\gamma\in R$.  Hence if $L/K$ is a semistable
extension then the $(\rho,k)$-refined breaks of $L/K$ are
equal to the SS$_k$-refined breaks for $2\le k\le p$.
In particular, the sets $B_{\rho,k}$ are independent of
$\rho$, so the VC$_k$-refined breaks of $L/K$ are
defined in this case.
\end{remark}

     Let $\gamma\in1+A$.  Then
$\ell_p(\gamma)=\mu\cdot(\gamma-1)$ for some
$\mu\in1+A$.  Hence for $x\in L$ we have
$v_L((\gamma-1)(x))=v_L(\ell_p(\gamma)(x))$.  It follows
that $\hat{v}_L(\gamma-1)=\hat{v}_L(\ell_p(\gamma))$,
and hence that $\hat{v}_L(\overline{\gamma-1})
=\hat{v}_L(\Lambdabar_p(\gammabar))$.  Therefore
\[\Gbar^{[\F]}_h=\{\gammabar\in\Gbar^{[\F]}:
\Lambdabar_p(\gammabar)\in\Jbar_h\}\]
is an $\F$-subspace of $(1+A)/(1+A^k)$
for all $h\ge0$.  It follows that the set of
SS$_k$-refined ramification breaks of $L/K$ is
\begin{equation} \label{set}
E_k=\{\hat{v}_L(\deltabar):\deltabar\in
\Span_{\F}(\Lambdabar_p(\Gbar)),
\;\deltabar\not=\overline{1}\}.
\end{equation}
We define the multiplicity of an SS$_k$-refined break
$h$ to be the $\F$-dimension of
$\Gbar^{[\F]}_h/\Gbar^{[\F]}_{h+1}$.  Since
$\Gbar^{[\F]}=\Gbar_0^{[\F]}$ has dimension $n$
over $\F$, the sum of the multiplicities of the
SS$_k$-refined breaks of $L/K$ is equal to $n$.

\begin{remark}
It follows from the above that $|E_k|\le n$, but it's not
obvious why $|E_k|=n$ should hold.  On the other hand, we
saw that if the VC$_k$-refined ramification breaks of
$L/K$ are defined then there are $n$ distinct
VC$_k$-refined breaks.
\end{remark}

\begin{remark}
Suppose $k=2$.  It follows from Corollary~\ref{log} that
the map
\[\Lambdabar_p:(1+A)/(1+A^2)\lra A/A^2\]
induced by $\Lambda_p:1+A\ra A$ is an isomorphism of
vector spaces over $\F$.  Hence $\Lambdabar_p(\Gbar)$
spans the $n$-dimensional $\F$-vector space $A/A^2$.  It
follows that the set of SS$_2$-refined breaks of $L/K$
is $\{\hat{v}_L(\deltabar):\deltabar\in A/A^2,\;
\deltabar\not=\overline{0}\}$.  Therefore the
SS$_2$-refined breaks of $L/K$ can be defined without
recourse to truncated powers or the truncated logarithm.
\end{remark}

\begin{remark} \label{ineq}
Let $2\le\ell\le k$ and let $\gammabar=\gamma+A^k\in
\Gbar^{[\F]}$.  Then
$(\gamma-1)+A^k\subset(\gamma-1)+A^{\ell}$, so we
have $\hat{v}_L((\gamma-1)+A^k)\le
\hat{v}_L((\gamma-1)+A^{\ell})$.  Therefore if we
arrange the SS$_k$-refined breaks and the
SS$_{\ell}$-refined breaks (counted with multiplicities)
in nondecreasing order then the SS$_k$-refined breaks
are less than or equal to the corresponding
SS$_{\ell}$-refined breaks.  A similar argument shows
that if $\rho\in L$ satisfies
$v_L(\rho)\equiv b\pmod{p^n}$ then the $(\rho,k)$-refined
breaks are less than or equal to the
$(\rho,\ell)$-refined breaks.  It follows that if the
VC$_k$-refined breaks and the VC$_{\ell}$-refined breaks
are defined then the VC$_k$-refined breaks are less
than or equal to the VC$_{\ell}$-refined breaks.
Finally, it follows from Definitions \ref{VC} and
\ref{SS} that the SS$_k$-refined breaks are less than or
equal to the $(\rho,k)$-refined breaks and the
VC$_k$-refined breaks.
\end{remark}

     We wish to give upper bounds for the refined breaks
of $L/K$.  We need the following well-known fact (see
for instance \cite[III, Prop.\,1.4]{cl}).

\begin{lemma} \label{Tr}
Let $L/K$ be a finite separable totally ramified
extension of local fields and let $\M_L^{d_{L/K}}$ be the
different of $L/K$.  Let $r\in\Z$.  Then
$\Tr_{L/K}(\M_L^r)=\M_K^s$, where
$s=\dst\left\lfloor\frac{r+d_{L/K}}{[L:K]}\right\rfloor$.
\end{lemma}

\begin{prop} \label{valTr}
Let $L/K$ be a finite separable totally ramified
extension of local fields and let $M/K$ be a
subextension of $L/K$.  Let $\M_L^{d_{L/K}}$ be the
different of $L/K$, let $\M_L^{d_{L/M}}$ be the
different of $L/M$, and let $\M_M^{d_{M/K}}$ be the
different of $M/K$.  Let $\rho\in L$ satisfy
$v_L(\rho)=-d_{L/K}-1$.  Then
\[v_M(\Tr_{L/M}(\rho))=\frac{d_{L/M}-d_{L/K}}{[L:M]}-1
=-d_{M/K}-1.\]
\end{prop}

\begin{proof}
Set $m=[L:M]$.  By Lemma~\ref{Tr} we have
$\Tr_{L/M}(\M_L^{-d_{L/K}-1})=\M_M^s$ and
$\Tr_{L/M}(\M_L^{-d_{L/K}})=\M_M^{s'}$ with
\begin{align*}
s&=\dst\left\lfloor\frac{d_{L/M}-d_{L/K}-1}{m}
\right\rfloor \\
s'&=\dst\left\lfloor\frac{d_{L/M}-d_{L/K}}{m}
\right\rfloor.
\end{align*}
Since $\M_L^{d_{L/K}}=\M_L^{d_{L/M}}\cdot\M_M^{d_{M/K}}$
we get $d_{L/K}-d_{L/M}=md_{M/K}$.  It follows that
$s=-d_{M/K}-1$ and $s'=-d_{M/K}$.  Hence $\Tr_{L/M}$
induces an isomorphism of $\OO_M$-modules
\[\M_L^{-d_{L/K}-1}/\M_L^{-d_{L/K}}\cong
\M_M^{-d_{M/K}-1}/\M_M^{-d_{M/K}}.\]
Since $\rho+\M_L^{-d_{L/K}}$ generates
$\M_L^{-d_{L/K}-1}/\M_L^{-d_{L/K}}$ as an
$\OO_M$-module, it follows that
$\Tr_{L/M}(\rho)+\M_M^{-d_{M/K}}$ generates
$\M_M^{-d_{M/K}-1}/\M_M^{-d_{M/K}}$ as an
$\OO_M$-module.  We conclude that
$v_M(\Tr_{L/M}(\rho))=-d_{M/K}-1$.
\end{proof}

\begin{prop} \label{bound}
Let $K$ be a local field of characteristic $p$ and let
$L/K$ be a totally ramified $C_p^n$-extension with a
single ramification break $b$.  Let $\rho\in L$ satisfy
$v_L(\rho)\equiv b\pmod{p^n}$, let $2\le k\le p$,
and let $b_0<b_1<\ldots<b_{n-1}$ be the $(\rho,k)$-refined
breaks of $L/K$.  Then for $0\le i<n$ we have
$b_i\le bp^i$.
\end{prop}

\begin{proof}
By Remark~\ref{ineq} it suffices to prove the proposition in the case
$k=2$.  For $1\le i\le n$ let $\Psi_i\in A$ be such that
$v_L(\Psi_i(\rho))-v_L(\rho)=b_{n-i}$.  Since $b_0,\ldots,b_{n-1}$ are
distinct the images of $\Psi_1,\ldots,\Psi_n$ in $A/A^2$ form an
$\F$-basis for $A/A^2$.  Suppose $b_j>bp^j$.  By the Steinitz
Exchange Lemma
there are $\tau_1,\ldots,\tau_j\in G$ such that
the images in $A/A^2$ of $\Psi_1,\ldots,\Psi_{n-j},
\tau_1-1,\ldots,\tau_j-1$ form a basis for $A/A^2$.
Let $H\cong C_p^j$ be the subgroup of $G$ generated by
$\tau_1,\ldots,\tau_j$ and let $M=L^H$ be the fixed
field of $H$.  Let $\A_H$ denote the augmentation
ideal of $\F[G/H]$, and observe that the images of
$\Psi_1,\ldots,\Psi_{n-j}$ in $\A_H/\A_H^2$ form a basis for
$\A_H/\A_H^2$.

     For $c\in K^{\times}$ and $\Upsilon\in A$ we have
\[v_L(\Upsilon(c\rho))-v_L(c\rho)
=v_L(\Upsilon(\rho))-v_L(\rho).\]
Hence the $(c\rho,k)$-refined breaks of $L/K$ are the
same as the $(\rho,k)$-refined breaks.  Therefore we may
assume that $v_L(\rho)=b-(b+1)p^n=-d_{L/K}-1$.  Then for
$1\le i\le n-j$ we have
\[v_L(\Psi_i(\rho))\ge bp^j+1+v_L(\rho)=bp^j-d_{L/K}.\]
Hence by Proposition~\ref{valTr} and Lemma~\ref{Tr} we
get
\begin{align*}
v_M(\Psi_i(\Tr_{L/M}(\rho)))-v_M(\Tr_{L/M}(\rho))
&=v_M(\Tr_{L/M}(\Psi_i(\rho)))-
\left(\frac{d_{L/M}-d_{L/K}}{p^j}-1\right) \\
&\ge\left\lfloor\frac{bp^j-d_{L/K}+d_{L/M}}{p^j}
\right\rfloor-\frac{d_{L/M}-d_{L/K}}{p^j}+1 \\
&=b+1.
\end{align*}
Since the images of $\Psi_1,\ldots,\Psi_{n-j}$ span
$\A_H/\A_H^2$ over $\F$, and
\[v_M(\Tr_{L/M}(\rho))=-(b+1)(p^{n-j}-1)-1\]
is not divisible by $p$, it follows that the lower
ramification breaks of $M/K$ are all $\ge b+1$.  Since
the only lower ramification break of $M/K$ is $b$, this
is a contradiction.  Hence we must have $b_i\le bp^i$
for $0\le i\le n-1$.
\end{proof}

\begin{cor} \label{min}
Let $K$ be a local field of characteristic $p$ and let
$L/K$ be a totally ramified $C_p^n$-extension with a
single ramification break $b$.  Then for $2\le k\le p$
the SS$_k$-refined breaks
$b_0\le b_1\le\cdots\le b_{n-1}$ of $L/K$ satisfy
$b_i\le bp^i$ for $0\le i\le n-1$.  If the
VC$_k$-refined breaks $b_0'<b_1'<\cdots<b_{n-1}'$ of
$L/K$ are defined they satisfy $b_i'\le bp^i$ for
$0\le i\le n-1$.
\end{cor}

\begin{proof}
This follows from the proposition and Remark~\ref{ineq}.
\end{proof}

\section{Scaffolds} \label{class}

In \cite[Theorem 18]{necessity}, it was observed that when the
VC$_p$-refined breaks attain the natural upper bounds given in
Proposition~\ref{bound}, the elements $\{\gamma_i\}$ which achieve
these bounds can be used to determine Galois module structure. These
elements $\gamma_i$ motivated a construction in \cite{scaf} referred
to as a ``Galois scaffold''. The properties of this Galois scaffold
led to the general definition of scaffold in \cite{bce}.  In this section, we
return to the construction in \cite{scaf}, but, as our aim is to study
the VC$_k$- and SS$_k$-refined breaks of these
extensions, we restrict
our attention to those extensions with only one ramification break.

Let $L/K$ be a $C_p^n$-extension with a single ramification break
$b>0$. As observed in section \ref{ref-brk-defn}, there is
Artin-Schreier data $(\beta,\vomega,\vepsilon)$ such that
$L=K(x_1,\dots,x_n)$, where $x_i\in K^{sep}$ is a root of the
polynomial $X^p-X-\omega_i^{p^n}\beta-\epsilon_i$ with
$v_K(\beta)=-b<v_K(\epsilon_i)$ and $\omega_i\in\F$. Recall that
$\omega_1,\ldots,\omega_n$ are linearly independent over $\F_p$.
We now consider, for each $1\leq r\leq n$, the restriction: for all
$1\le i\le n$,
\begin{equation}\label{filt}
  v_K(\epsilon_i)>-b/p^{n-r}.
\end{equation}
At one extreme, $r=n$, this is no additional restriction. At the other extreme, $r=1$, a Galois scaffold exists.

\subsection{The case $r=1$}
Observe that \eqref{filt} with $r=1$ 
is precisely 
Assumption 3.3 in \cite{large} 
for extensions with one ramification break $b>0$.
As a result, these extensions possess a Galois scaffold.  The original
construction of a Galois scaffold in \cite{scaf} can be
broken into two separate parts, as was done in
\cite{large}.  In
\cite[\S3]{large}, field elements of nice valuation are constructed
upon which the Galois action is easily described.  In
\cite[\S2]{large}, these elements and the nice description of the
Galois action are used to construct the two ingredients of a scaffold:
$\lambda_w\in L$ with $v_L(\lambda_w)=w$ for all $w\in\Z$, and
$\Psi_i\in K[G]$ for $1\le i\le n$ such that $\Psi_i(\lambda_w)\in L$
is congruent either to $\lambda_x$ for some $x\in\Z$, or
to 0.  In this section, we
introduce a method that allows us to more easily construct the field
elements of nice valuation constructed by \cite[\S3]{large}. Namely, we
construct $Y\in L$ such that $v_L(Y)=-b$ and $(\sigma-1)(Y)\in\F$ for
all $\sigma\in\Gal(L/K)$.  Since $p\nmid b$ the condition $v_L(Y)=-b$
implies $L=K(Y)$. We reference \cite[\S2]{large} for the construction of the rest
of the ingredients of the Galois scaffold.

Let $\x\in L^n$ be the column vector whose
$i$th entry is $x_i$ and define the Frobenius
endomorphism $\phi:L\ra L$ by $\phi(\alpha)=\alpha^p$.
Then $\phi(\x)=\x+\beta\phi^n(\vomega)+\vepsilon$.  Let
\begin{equation} \label{Y}
Y=\det([\x,\phi(\vomega),\phi^2(\vomega),\dots,
\phi^{n-1}(\vomega)]).
\end{equation}
By expanding in cofactors along the first column we get
\begin{equation} \label{cofactor}
Y=t_1x_1+t_2x_2+\cdots+t_nx_n,
\end{equation}
with $t_i\in\F$.  Let $\sigma\in G$ and set
$u_i=\sigma(x_i)-x_i\in\F_p$.  Then
\begin{equation} \label{sYY}
\sigma(Y)-Y=t_1u_1+t_2u_2+\cdots+t_nu_n\in\F.
\end{equation}

     Let $1\le i\le n$.  Since
$\omega_1,\ldots,\omega_{i-1},\omega_{i+1},\ldots,\omega_n$
are linearly independent over $\F_p$, the following
lemma implies $t_i\not=0$.

\begin{lemma} \label{nonz}
Let $\alpha_1,\ldots,\alpha_d\in\F$ be linearly
independent over $\F_p$, and let $\vec{\alpha}$ be the
column vector with $d$ entries whose $i$th entry is
$\alpha_i$.  Then
\[\det([\vec{\alpha},\phi(\vec{\alpha}),
\phi^2(\vec{\alpha}),\dots,\phi^{d-1}(\vec{\alpha})])
\not=0.\]
\end{lemma}

\begin{proof}
Since $\alpha_1,\ldots,\alpha_d\in\F$ are linearly
independent over $\F_p$, this Moore determinant is
nonzero \cite[Lemma~1.3.3]{goss}.
\end{proof}

\begin{prop} \label{vLY}
We have $v_L(Y)=-b$, and hence $L=K(Y)$.
\end{prop}

\begin{proof} We claim that for $0\le j\le n-1$ we have
\begin{align} \label{phijY}
\phi^j(Y)&\equiv\det([\x,\phi^{j+1}(\vomega),
\phi^{j+2}(\vomega),\dots,\phi^{j+n-1}(\vomega)])
\pmod{\M_L^{-bp^j+1}}.
\end{align}
The claim holds for $j=0$ by the definition of $Y$.  Let
$1\le j\le n-1$ and assume the claim holds for $j-1$.
Observe that $\phi(\x)=\x+\beta\phi^n(\vomega)+\vepsilon$ and that, because $r=1$,
$v_L(\epsilon_i)\ge-bp+1\ge-bp^j+1$. Therefore we get
\begin{alignat*}{2}
\phi^j(Y)&\equiv\phi(\det([\x,\phi^j(\vomega),
\phi^{j+1}(\vomega),\dots,\phi^{j+n-2}(\vomega)]))
&&\pmod{\M_L^{p(-bp^{j-1}+1)}} \\
&\equiv\det([\x+\beta\phi^n(\vomega)+\vepsilon,
\phi^{j+1}(\vomega),\phi^{j+2}(\vomega),\dots,
\phi^{j+n-1}(\vomega)])&&\pmod{\M_L^{-bp^j+1}} \\
&\equiv\det([\x+\beta\phi^n(\vomega),
\phi^{j+1}(\vomega),\phi^{j+2}(\vomega),\dots,
\phi^{j+n-1}(\vomega)])&&\pmod{\M_L^{-bp^j+1}}.
\end{alignat*}
Since $j+1\le n\le j+n-1$ it follows that
\[\phi^j(Y)\equiv\det([\x,\phi^{j+1}(\vomega),
\phi^{j+2}(\vomega),\dots,\phi^{j+n-1}(\vomega)])
\pmod{\M_L^{-bp^j+1}}.\]
Therefore by induction the claim holds for
$0\le j\le n-1$.  The same reasoning with $j=n$ gives
\[\phi^n(Y)\equiv\det([\x+\beta\phi^n(\vomega),
\phi^{n+1}(\vomega),
\phi^{n+2}(\vomega),\dots,\phi^{2n-1}(\vomega)])
\pmod{\M_L^{-bp^n+1}}.\]
Since $v_L(x_i)=-bp^{n-1}$ we get
\begin{alignat*}{2}
\phi^n(Y)&\equiv\det([\beta\phi^n(\vomega),
\phi^{n+1}(\vomega),
\phi^{n+2}(\vomega),\dots,\phi^{2n-1}(\vomega)])
&&\pmod{\M_L^{-bp^n+1}} \\
&\equiv\beta\det([\phi^n(\vomega),\phi^{n+1}(\vomega),
\phi^{n+2}(\vomega),\dots,\phi^{2n-1}(\vomega)])
&&\pmod{\M_L^{-bp^n+1}}.
\end{alignat*}
It follows from Lemma~\ref{nonz} that
$v_L(Y^{p^n})=v_L(\beta)$, and hence that
$p^nv_L(Y)=-bp^n$.  Thus $v_L(Y)=-b$.  Since $p\nmid b$
this implies $L=K(Y)$.
\end{proof}

     The main result of this section, Theorem \ref{scaffold}, says that the
extension $L/K$ possesses a Galois scaffold.  To
give the definition of a Galois scaffold, using notation consistent with \cite{large, bce}, we first
define $\a:\Z\ra\Z/p^n\Z$ by setting
$\a(t)=-b^{-1}t+p^n\Z$, where $b^{-1}$ denotes the
multiplicative inverse of the class of $b$ in
$\Z/p^n\Z$.  We
then express $\a(t)$ in base $p$ by writing
\[\a(t)=(\a(t)_{(0)}p^0+\a(t)_{(1)}p^1+\cdots
+\a(t)_{(n-1)}p^{n-1})+p^n\Z\]
with $0\le \a(t)_{(i)}<p$.  Specializing
Definition~2.3 of \cite{bce} to our setting we get:

\begin{definition}
Let $L/K$ be a totally ramified $C_p^n$-extension of
local fields with a single ramification break $b$.  A
Galois scaffold for $L/K$ with infinite precision
consists of elements $\lambda_w\in L$ for all $w\in\Z$
and $\Psi_i\in K[G]$ for $1\le i\le n$ such that the
following hold:
\begin{enumerate}[(i)]
\item $v_L(\lambda_w)=w$ for all $w\in\Z$.
\item $\lambda_{w_1}\lambda_{w_2}^{-1}\in K$
whenever $w_1\equiv w_2\pmod{p^n}$.
\item $\Psi_i\cdot1=0$ for $1\le i\le n$.
\item For $1\le i\le n$ and $w\in\Z$ there exists
$u_{iw}\in\OO_K^{\times}$ such that the following
holds:
\[\Psi_i(\lambda_w)=\begin{cases} 
u_{iw}\lambda_{w+p^{n-i}b}&
\mbox{if }\a(w)_{(n-i)}\ge1, \\ 
0&\mbox{if }\a(w)_{(n-i)}=0.
\end{cases} \]
\end{enumerate}
\end{definition}

\begin{theorem} \label{scaffold}
Let $L/K$ be a $C_p^n$-extension with a single ramification break $b$.
Assume there is Artin-Schreier data $(\beta,\vomega,\vepsilon)$ for
$L/K$ such that $v_K(\epsilon_i)>-b/p^{n-1}$ for $1\le i\le n$.  Then
the extension $L/K$ has a Galois scaffold $(\{\lambda_w\},\{\Psi_i\})$
with infinite precision such that $u_{iw}=1$ for all $i,w$ and $\Psi_i\in\F[G]$ for $1\le i\le n$.
Furthermore, the image of $1+\Psi_i$ in
$\Rbar=\F[G]/A^p$ lies in $\Gbar^{[\F]}$.
\end{theorem}

\begin{proof}
For $1\le i\le n$ let $x_i\in K^{sep}$ be a root of
$X^p-X-\omega_i^{p^n}\beta-\epsilon_i$.  For $0\le i\le n$ set
$K_i=K(x_1,\ldots,x_i)$, so that $K=K_0$ and $L=K_n$.  Let
$\sigma_1,\ldots,\sigma_n$ be generators for $G=\Gal(L/K)$ such that
$(\sigma_i-1)(x_j)=\delta_{ij}$ for $1\le i,j\le n$.  For $1\le i\le
n$ construct a generator $Y_i$ for the extension $K_i/K$ as in
(\ref{Y}).  For $1\le j\le i$ set $\mu_{ij}=(\sigma_j-1)(Y_i)$.  Then
$\mu_{ij}\not=0$ since $K_i=K(Y_i)$, and by (\ref{sYY}) we have
$\mu_{ij}\in\F$.  Set $X_i=Y_i/\mu_{ii}$.  Then $K_i=K_{i-1}(X_i)$ and
$(\sigma_i-1)(X_i)=1$.  It follows from Theorem~2.10 in \cite{large}
that $L/K$ has a Galois scaffold $(\{\lambda_w\},\{\Psi_i\})$ with
infinite precision such that $u_{iw}=1$ for all $i,w$. (This result appeared
first as Corollary~4.2 in \cite{scaf}.)  Since $\mu_{ij}\in\F$ it
follows from Definition~2.7 in \cite{large} that $\Psi_i\in\F[G]$ and
$(1+\Psi_i)+A^p\in\Gbar^{[\F]}$ for $1\le i\le n$.
\end{proof}

Let $(\{\lambda_w\},\{\Psi_i\})$ be the scaffold for
$L/K$, and let $\rho$ be a valuation criterion element
for $L/K$.  Then $v_L(\rho)\equiv b\pmod{p^n}$, so we
have $\a(v_L(\rho))=(p^n-1)+p^n\Z$.  For $0\le t<p^n$
write $t=t_{(0)}+t_{(1)}p+\cdots+t_{(n-1)}p^{n-1}$
with $0\le t_{(i)}<p$ and define
$\Psi^{(t)}=\Psi_n^{t_{(0)}}\Psi_{n-1}^{t_{(1)}}
\ldots\Psi_1^{t_{(n-1)}}$.  Using induction we see that
for $0\le t<p^n$ we have
\begin{align}
v_L(\Psi^{(t)}(\rho))&=v_L(\rho)+bt \label{Psit} \\
\a(v_L(\Psi^{(t)}(\rho)))&=(p^n-1-t)+p^n\Z, \nonumber
\end{align}
while for general $\alpha\in L$, we have $v_L(\Psi^{(t)}(\alpha))\geq v_L(\alpha)+bt$.

\subsection{The case $r>1$}
The existence of a Galois scaffold, or even a partial Galois scaffold,
can be used to determine the values of VC$_k$- and SS$_k$-refined
breaks.  In this section we examine conditions that produce partial
 Galois scaffolds.  To begin we need to look more closely at
$C_p^2$-extensions.  The following lifting lemma will enable us to
lift Artin-Schreier data, in particular $\beta\in K$ and $\vepsilon\in
K^n$, up into $K(x_1)$ and $K(x_1)^n$, respectively. This lemma is
crucial, both here and in the next section.

\begin{lemma} \label{zE} 
Let $M/K$ be a totally ramified $C_p^2$-extension with a
single ramification break $b$ and Artin-Schreier data
$(\beta,\vomega,\vepsilon)$.
Assume without loss of generality that
$\omega_1=1$ and $\epsilon_1=0$, so that $M=K(x_1,x_2)$
with $x_1^p-x_1=\beta$ and $x_2^p-x_2
=\omega_2^{p^2}\beta+\epsilon_2$.  Then
the following
hold:
\begin{enumerate}[(a)]
\item There exist $\z,E\in K(x_1)$ such that
$\z^p-\z=E-\epsilon_2$ and
$v_{K(x_1)}(\z)=v_{K(x_1)}(E)=-e_2$,  where
$e_2=-v_K(\epsilon_2)$.  Furthermore,
$v_{K(x_1)}(E-\wp(\omega_2)^px_1)=-b$.
\item Let $\z,E$ be as in (a) and set
$\X=x_2-\omega_2^px_1+\z$.  Then
$\X^p-\X=-\wp(\omega_2)^px_1+E$ and $M=K(\X)$.
\end{enumerate}
\end{lemma}

\begin{proof}
(a) Let $z\in K^{sep}$ satisfy $z^p-z=\epsilon_2$.  By
the definition of Artin-Schreier data we have either
$e_2>0$ and $p\nmid e_2$, or $\epsilon_2\in\F$.
Suppose $e_2>0$ and $p\nmid e_2$.  Then $b,e_2$ are the
upper ramification breaks of $K(x_1,z)/K$.  Since
$e_2<b$ it follows that $e_2$ is also a lower
ramification break of $K(x_1,z)/K$.  Hence
$K(x_1,z)/K(x_1)$ is a totally ramified $C_p$-extension
with ramification break $e_2$.  Hence by Artin-Schreier
theory there are $\z,E\in K(x_1)$ such that
$(z+\z)^p-(z+\z)=E$ and
$v_{K(x_1)}(\z)=v_{K(x_1)}(E)=-e_2$.  If
$\epsilon_2\in\F^{\times}$ we set $\z=1$ and
$E=\epsilon_2$.  Then
$v_{K(x_1)}(\z)=v_{K(x_1)}(E)=0=-e_2$.  If
$\epsilon_2=0$ we set $\z=E=0$.  Then
$v_{K(x_1)}(\z)=v_{K(x_1)}(E)=\infty=-e_2$.  Hence
$\z^p-\z=E-\epsilon_2$ and
$v_{K(x_1)}(\z)=v_{K(x_1)}(E)=-e_2$ in all cases.  In
addition, since $1,\omega_2$ are linearly independent
over $\F_p$ we have $\omega_2\not\in\F_p$, and hence
$\wp(\omega_2)\not=0$.  Since
\[v_{K(x_1)}(E)=-e_2>-b=v_{K(x_1)}(x_1)\]
it follows that $v_{K(x_1)}(E-\wp(\omega_2)^px_1)=-b$.
\\[\smallskipamount]
(b) By the definition of $\X$ we get
\begin{align*}
\X^p-\X&=x_2^p-x_2-\omega_2^{p^2}x_1^p+\omega_2^px_1
+\zeta^p-\zeta \\
&=\omega_2^{p^2}\beta+\epsilon_2-\omega_2^{p^2}\beta
-\omega_2^{p^2}x_1+\omega_2^px_1+E-\epsilon_2 \\
&=-\wp(\omega_2)^px_1+E.
\end{align*}
Since $\X\in M$ and $v_M(\X)=-b$ with $p\nmid b$ it
follows that $M=K(\X)$.
\end{proof}

\begin{prop} \label{LK1}
Let $L/K$ be a $C_p^n$-extension with a single
ramification break $b$.  Let $0<c<b$, and assume that there
exists Artin-Schreier data $(\beta,\vomega,\vepsilon)$
for $L/K$ such that $v_K(\epsilon_i)>-c$ for
$1\le i\le n$.  Let $x_1\in K^{sep}$ satisfy
$x_1^p-x_1=\omega_1^{p^n}\beta+\epsilon_1$, and set
$K_1=K(x_1)$.  Then there is Artin-Schreier data
$(x_1,\vOmega,\vE)$ for $L/K_1$ such that
$v_{K_1}(E_i)>-c$ for $2\le i\le n$.
\end{prop}

\begin{proof}
By replacing $\beta$ with
$\beta'=\omega_1^{p^n}\beta+\epsilon_1$ me may assume
without loss of generality that $\omega_1=1$ and
$\epsilon_1=0$.  For $2\le i\le n$ let $x_i\in K^{sep}$
satisfy $x_i^p-x_i=\omega_i^{p^n}\beta+\epsilon_i$.  By
applying Lemma~\ref{zE} to $K(x_1,x_i)/K$ we get
$\z_i,E_i\in K_1$ such that $\z_i^p-\z_i=E_i-\epsilon_i$
and $v_{K_1}(\z_i)=v_{K_1}(E_i)=v_L(\epsilon_i)$.
Set $\X_i=x_i-\omega_i^{p^{n-1}}x_1+\z_i$.  Then
$K_1(\X_i)=K_1(x_i)$, so we have
\[L=K_1(x_2,\dots,x_n)=K_1(\X_2,\ldots,\X_n).\]
We also have
$\X_i^p-\X_i=-\wp(\omega_i)^{p^{n-1}}x_1+E_i$ with
$v_{K_1}(x_1)=-b$ and $v_{K_1}(E_i)=v_L(\epsilon_i)>-c$.
Since $\Span_{\F_p}\{1,\omega_2,\ldots,\omega_n\}$ has
$\F_p$-dimension $n$, and $\wp:\F\ra\F$ is an
$\F_p$-linear map with kernel $\F_p$, we see that
$\wp(\omega_2),\ldots,\wp(\omega_n)$ are linearly
independent over $\F_p$.  Setting
\[\vOmega=\begin{bmatrix}-\omega_2^{p^{n-1}}\\\vdots\\
-\omega_n^{p^{n-1}}\end{bmatrix}\hspace{1cm}
\vE=\begin{bmatrix}E_2\\\vdots\\E_n\end{bmatrix}\]
we deduce that $(x_1,\vOmega,\vE)$ is Artin-Schreier
data for $L/K_1$.
\end{proof}

\begin{cor} \label{LKr}
Let $L/K$ be a $C_p^n$-extension with a single
ramification break $b$.  Let $0<c<b$, and assume that there
exists Artin-Schreier data $(\beta,\vomega,\vepsilon)$
for $L/K$ such that $v_K(\epsilon_i)>-c$ for
$1\le i\le n$.  Let $1\le r\le n-1$, and for $1\le i\le r$
let $x_i\in K^{sep}$ satisfy
$x_i^p-x_i=\omega_i^{p^n}\beta+\epsilon_i$.  Set
$K_r=K(x_1,\ldots,x_r)$.  Then there is Artin-Schreier
data $(x_r,\vOmega,\vE)$ for $L/K_r$ such that
$v_{K_r}(E_i)>-c$ for $r+1\le i\le n$.
\end{cor}

\begin{prop} \label{scafr}
Let $L/K$ be a $C_p^n$-extension with a single
ramification break $b$.  Let $1\le r\le n-1$, and assume
that there exists Artin-Schreier data
$(\beta,\vomega,\vepsilon)$
for $L/K$ such that $v_K(\epsilon_i)>-b/p^{n-1-r}$ for
$1\le i\le n$.  For $1\le i\le r$ let $x_i\in K^{sep}$
satisfy $x_i^p-x_i=\omega_i^{p^n}\beta+\epsilon_i$, and
set $K_r=K(x_1,\ldots,x_r)$.
Let $G_r=\Gal(L/K_r)$ and let $A_r$ be the augmentation
ideal of $R_r=\F[G_r]$.  Then $L/K_r$ has a Galois
scaffold $(\{\lambda_w\},\{\Psi_i\}_{r+1\le i\le n})$ with
infinite precision such that $\Psi_i\in R_r$ and the
image of $1+\Psi_i$ in $\Rbar_r=R_r/A_r^p$ lies in
$\Gbar_r^{[\F]}$ for $r+1\le i\le n$.
\end{prop}

\begin{proof}
By Corollary~\ref{LKr} there is Artin-Schreier data
$(x_r,\vOmega,\vE)$ for $L/K_r$ such that
$v_{K_r}(E_i)>-b/p^{n-1-r}$ for $r+1\le i\le n$.
Since $[L:K_r]=p^{n-r}$ it follows from
Theorem~\ref{scaffold} that $L/K_r$ has a Galois
scaffold $(\{\lambda_w\},\{\Psi_i\}_{r+1\le i\le n})$
with the specified properties.
\end{proof}

\begin{remark}
Suppose there is Artin-Schreier data
$(\beta,\vomega,\vepsilon)$ for $L/K$ which satisfies
the hypotheses of Proposition~\ref{scafr}.  Let $L/K_r'$
be a subextension of $L/K$ such that $[K_r':K]=p^r$.
Then there is Artin-Schreier data
$(\beta,\vomega',\vepsilon\,')$ for $L/K$ such that:
\begin{enumerate}
\item There is $A\in\GL_n(\F_p)$ such that
$\vomega'=A\vomega$ and $\vepsilon\,'=A\vepsilon$.
\item $K_r'=K(x_1',\ldots,x_r')$ with
$x_i'^p-x_i'=\omega_i'^{p^n}\beta+\epsilon_i'$ for
$1\le i\le r$.
\end{enumerate}
It follows that $(\beta,\vomega',\vepsilon\,')$ also
satisfies the hypotheses of Proposition~\ref{scafr}.
Therefore the conclusion of Proposition~\ref{scafr} holds
for $L/K_r'$.
\end{remark}

\section{Computing refined breaks} \label{comp}

Let $L/K$ be a totally ramified $C_p^n$-extension with a
single ramification break $b$, and set $G=\Gal(L/K)$.
In \cite[Lemma~3]{necessity} it was observed that when
$n=2$ the extension $L/K$ has the following property:
There is a subgroup $H\le G$ with index $p$ such that
$L/L^H$ has a Galois scaffold consisting of elements of
$K[G]$.  This property is used in \cite{necessity} to
prove that the values of the $(\rho,p)$-refined breaks
are independent of the choice of valuation criterion
element $\rho$.  Now suppose $\ch(K)=p$ and $n\ge2$.  It
follows from Proposition~\ref{scafr} that if $L/K$ has
Artin-Schreier data satisfying \eqref{filt} with $r=2$
then $L/K$ also has this property.  In this section we
use this observation to show that for this family of
extensions the VC$_k$- and SS$_k$-refined breaks are
equal and independent of $k$.

\subsection{An equivalence condition for SS$_k$- and VC$_k$-refined breaks}

 To prove our first main result we need two basic
lemmas.

\begin{lemma} \label{aug}
Let $L/K$ be a totally ramified $C_p^n$-extension with a
single ramification break $b>0$.  Assume that $L/K$ has a Galois
scaffold $(\{\lambda_w\},\{\Psi_i\})$ with infinite
precision such that $\Psi_i\in\F[G]$ for
$1\le i\le n$.  Then the augmentation ideal $A$ of
$\F[G]$ is generated by $\Psi_1,\ldots,\Psi_n$.
\end{lemma}

\begin{proof}
Let $I$ denote the ideal in $\F[G]$ generated by
$\Psi_1,\ldots,\Psi_n$, and let $\I$ denote the ideal in
$K[G]$ generated by $\Psi_1,\ldots,\Psi_n$.  Let $\A$
denote the augmentation ideal of $K[G]$.  Then the
isomorphism $\F[G]\otimes_{\F}K\cong K[G]$
induces isomorphisms $A\otimes_{\F}K\cong\A$ and
$I\otimes_{\F}K\cong\I$.  Therefore it suffices to
prove that $\A=\I$.  Let $\rho$ be a valuation criterion
element of $L/K$.  Then the map $\eta:K[G]\ra L$ defined
by $\eta(\gamma)=\gamma(\rho)$ is an isomorphism of
$K$-vector spaces.  For $1\le s<p^n$ we have
$v_L(\Psi^{(s)}(\rho))=v_L(\rho)+bs$.  Therefore the
elements of $\{\eta(\Psi^{(s)}):1\le s<p^n\}$ have
$L$-valuations which represent distinct congruence
classes modulo $p^n$.  Hence
$\dim_K(\I)=\dim_K(\eta(\I))\ge p^n-1$.  Since
$\I\subset\A$ and $\dim_K(\A)=p^n-1$ it follows that
$\I=\A$.
\end{proof}

\begin{lemma} \label{spell}
Let $\sigma\in G$ and let $\alpha,\beta\in L$.  Then for
$0\le r\le p-1$ we have
\[(\sigma-1)^r(\alpha\beta)=
\alpha\cdot(\sigma-1)^r(\beta)
+\sum_{i=0}^r(-1)^{r-i}\binom{r}{i}(\sigma^i-1)(\alpha)
\cdot\sigma^i(\beta).\]
\end{lemma}

\begin{proof}
We have
\begin{align*}
\sum_{i=0}^r(-1)^{r-i}\binom{r}{i}(\sigma^i-1)(\alpha)
\cdot\sigma^i(\beta)
&=\sum_{i=0}^r(-1)^{r-i}\binom{r}{i}\sigma^i(\alpha\beta)
-\alpha\sum_{i=0}^r(-1)^{r-i}\binom{r}{i}\sigma^i(\beta)
\\
&=(\sigma-1)^r(\alpha\beta)
-\alpha\cdot(\sigma-1)^r(\beta). \qedhere
\end{align*}
\end{proof}

\begin{theorem} \label{hatv}
Let $L/K$ be a $C_p^n$-extension with a single
ramification break $b>0$.  Assume that there is
Artin-Schreier data $(\beta,\vomega,\vepsilon)$ for
$L/K$ such that $v_K(\epsilon_i)>-b/p^{n-2}$ for
$1\le i\le n$.  Let $\rho$ be a valuation criterion
element for $L/K$.  Then for every $\Upsilon\in
A\smallsetminus A^2$ we have
$\hat{v}_L(\Upsilon)=v_L(\Upsilon(\rho))-v_L(\rho)$.
\end{theorem}

\begin{proof}
We may assume without loss of generality that
$\omega_1=1$ and $\epsilon_1=0$.  Let $x_1\in L$ satisfy
$x_1^p-x_1=\beta$, set $K_1=K(x_1)$, and let
$G_1=\Gal(L/K_1)$.  Then by Proposition~\ref{scafr} there
is a scaffold $(\{\lambda_w\},\{\Psi_i\}_{2\le i\le n})$
for $L/K_1$ with infinite precision such that
$\Psi_i\in\F[G_1]$ for $2\le i\le n$.
Choose
$\theta\in L^{\times}$ which minimizes
$v_L(\Upsilon(\theta))-v_L(\theta)$.  Since $p\nmid b$
there is $0\le s<p^{n-1}$ such that
\[bs\equiv v_L(\theta)-v_L(\rho)\pmod{p^{n-1}}.\]
Hence there is $a\in K_1$ such that
$v_L(\theta)=v_L(a)+bs+v_L(\rho)$.
Since $v_L(\rho)\equiv b\bmod p^{n-1}$, it follows from
\eqref{Psit} that
\[v_L(\Psi^{(s)}(\rho))=v_L(\rho)+bs.\]
Thus
$v_L(\theta)=v_L(a\Psi^{(s)}(\rho))$.
By Lemma~\ref{all} we have
\[v_L(\Upsilon(a\Psi^{(s)}(\rho)))-v_L(a\Psi^{(s)}(\rho))
=v_L(\Upsilon(\theta))-v_L(\theta).\]
Hence we may assume that $\theta=a\Psi^{(s)}(\rho)$.

     Let $\sigma\in G$ be such that $\sigma|_{K_1}$
generates $\Gal(K_1/K)$.  Then by Lemma~\ref{aug} we
have
\[\Upsilon=\sum_{r=0}^{p-1}\sum_{t=0}^{p^{n-1}-1}
c_{rt}(\sigma-1)^r\Psi^{(t)}\]
for some $c_{rt}\in\F$.  Using Lemma~\ref{spell} we get
\begin{align*}
(\sigma-1)^r(\Psi^{(t)}(\theta)) 
&=(\sigma-1)^r(a\Psi^{(t)}\Psi^{(s)}(\rho)) \\
&=a(\sigma-1)^r(\Psi^{(t)}\Psi^{(s)}(\rho))
+\sum_{i=0}^r(-1)^{r-i}\binom{r}{i}(\sigma^i-1)(a)
\cdot\sigma^i(\Psi^{(t)}\Psi^{(s)}(\rho)).
\end{align*}
For $0\le i\le r$ the first factor in the $i$th term in
the sum above has $L$-valuation at least
$v_L(a)+bp^{n-1}$, and the second factor has
$L$-valuation at least $bt+v_L(\Psi^{(s)}(\rho))$.
Hence the terms in the sum all have $L$-valuations at
least
\[v_L(a)+bp^{n-1}+bt+v_L(\Psi^{(s)}(\rho))
=bp^{n-1}+bt+v_L(\theta).\]
It follows that
\begin{alignat*}{2}
(\sigma-1)^r(\Psi^{(t)}(\theta))&\equiv
a(\sigma-1)^r(\Psi^{(t)}\Psi^{(s)}(\rho))
&&\pmod{\theta\cdot\M_L^{bp^{n-1}}}
\end{alignat*}
and hence that
\begin{alignat*}{2}
\Upsilon(\theta)&\equiv a\Upsilon(\Psi^{(s)}(\rho))
&&\pmod{\theta\cdot\M_L^{bp^{n-1}}} \\
&\equiv a\Psi^{(s)}(\Upsilon(\rho))
&&\pmod{\theta\cdot\M_L^{bp^{n-1}}}.
\end{alignat*}
Considering $\Upsilon(\rho)$ to be a generic element of $L$, we have
$v_L(\Psi^{(s)}(\Upsilon(\rho)))\geq v_L(\Upsilon(\rho))+bs$.  Therefore,
\begin{align*}
v_L(\Upsilon(\theta))-v_L(\theta)
&\ge\min\{v_L(a)+bs+v_L(\Upsilon(\rho))-
(v_L(a)+v_L(\Psi^{(s)}(\rho))),bp^{n-1}\} \\
&=\min\{v_L(\Upsilon(\rho))-v_L(\rho),bp^{n-1}\} \\
&=v_L(\Upsilon(\rho))-v_L(\rho),
\end{align*}
where the last equality follows from $\Upsilon\not\in A^2$ and the
case $k=2$ of Proposition~\ref{bound}.
 We conclude that $\hat{v}_L(\Upsilon)
=v_L(\Upsilon(\rho))-v_L(\rho)$.
\end{proof}

\begin{cor} \label{same}
Let $L/K$ be a $C_p^n$-extension with a single
ramification break $b$.  Assume there is Artin-Schreier
data $(\beta,\vomega,\vepsilon)$ for $L/K$ such that
$v_K(\epsilon_i)>-b/p^{n-2}$ for $1\le i\le n$.  Then
for $2\le k\le p$ the set of VC$_k$-refined breaks of
$L/K$ is defined and equal to the set of SS$_k$-refined
breaks of $L/K$.
\end{cor}

\subsection{Explicit computation of refined breaks}
By strengthening the assumption in Theorem \ref{hatv},
we can explicitly determine the values of the refined ramification breaks.
We will do so by following the process used earlier in \S3.1, as well as in
\cite[\S5]{large}.

\begin{theorem} \label{main}
Let $n\ge2$ and let $L/K$ be a $C_p^n$-extension with a
single ramification break $b$.  Let
$(\beta,\vomega,\vepsilon)$ be Artin-Schreier data for
$L/K$ such that $\omega_1=1$ and $\epsilon_1=0$.  For
$2\le i\le n$ set $e_i=-v_K(\epsilon_i)$, and assume that
$e_n<b/p^{n-2}$ and $e_i<e_n$ for $2\le i\le n-1$.  Let
\begin{align*}
b_*&=\min\{(b-e_n)p^{n-1}+b,bp^{n-1}\} \\
B_1&=\{b,bp,\ldots,bp^{n-2},b_*\}.
\end{align*}
Then for $2\le k\le p$, the set of VC$_k$-refined breaks
of $L/K$ is defined and equal to $B_1$, and the set of
SS$_k$-refined breaks of $L/K$ is equal to $B_1$.
\end{theorem}

\begin{remark}
Let $L/K$ be an extension which satisfies the hypotheses
of Theorem~\ref{main}.
Equation (4.20) in \cite{elem} gives a description of
$L/K$ in terms of certain parameters
$w_0,w_1,\ldots,w_{n-1}$ from $K$ satisfying
$v_K(w_i)\ge-b$ and $v_K(w_0)=-b$.  (Beware that
$w_i\not=\omega_i$.)
Assume that the
$\F_p$-span of $1=\omega_1,\omega_2,\ldots,\omega_n$
forms a subfield $\F_{p^n}$ of $\F$ with $p^n$ elements.
Then by Lemma~5.2 of \cite{elem} we
have $v_K(w_i)\ge-e_n$ for $1\le i\le n-1$.
Furthermore, using the fact that $e_i<e_n$ for
$2\le i\le n-1$ it can be shown that $v_K(w_{n-1})=-e_n$.  Assume
$b<p$, so that we can use Theorem~5.1 of \cite{elem} to
compute the indices of inseparability of $L/K$.  We get
$i_n=0$, $i_{n-1}=bp^n+\min\{-e_np^{n-1},-b\}$, and
$i_j=bp^n-b$ for $0\le j\le n-2$.  By applying
Theorem~\ref{main} we deduce that $i_j=b_j+bp^n-bp^j-b$
for $0\le i\le n-1$, where $b_0<b_1<\cdots<b_{n-1}$ are
the refined breaks of $L/K$.  We note that these results
are in agreement with the formulas relating indices of
inseparability and refined breaks for $C_p^2$-extensions
in characteristic 0 given in Theorem~4.6 of
\cite{refined}.
\end{remark}

     The proof of Theorem~\ref{main} will occupy the
rest of this section.  For any valuation criterion
element $\rho$ for $L/K$ it suffices, by Corollary~\ref{same},
to prove that for $2\le k\le p$,
 $B_1$ is the set of $(\rho,k)$-refined breaks of $L/K$.  For
 $\Upsilon\in\F[G]$ we define, depending upon $\rho$,
\begin{align*}
\hat{v}_{\rho}(\Upsilon)&=v_L(\Upsilon(\rho))-v_L(\rho) \\
\hat{v}_{\rho}(\Upsilon+A^k)&=\max\{\hat{v}_{\rho}(\Upsilon')
:\Upsilon'\in\Upsilon+A^k\}.
\end{align*}
Note that $\hat{v}_{\rho}(\Upsilon+A^k)=i_\rho(\overline{\Upsilon})$.
So to prove that $B_1$ is the set of $(\rho,k)$-refined
breaks for $2\le k\le p$ it is enough to first, construct
$\Psi_1,\ldots,\Psi_n\in\F[G]$ such that
\begin{equation} \label{B1}
B_1=\{\hat{v}_{\rho}(\Psi_i):1\le i\le n\}
\end{equation}
and second, prove
$\hat{v}_{\rho}(\Psi_i')\le \hat{v}_{\rho}(\Psi_i)$ for
all $\Psi_i'\in\Psi_i+A^2$.

\subsubsection{Construction of
$\Psi_1,\ldots,\Psi_n\in\F[G]$ satisfying \eqref{B1}}

We first separate off the case when $e_n<b/p^{n-1}$ (and
hence $b_*=bp^{n-1}$).  If
$e_n<b/p^{n-1}$ then by Theorem~\ref{scaffold} we get a scaffold
$(\{\lambda_w\},\{\Psi_i\})$ for $L/K$ such that $\Psi_i\in\F[G]$ for
$1\le i\le n$.  By (\ref{Psit}) we see that $\Psi_1,\ldots,\Psi_n$
satisfy (\ref{B1}).  Thus we may assume for the remainder of the
argument that $e_n\geq b/p^{n-1}>0$, which means, by the definition of
Artin-Schreier data, that we have $p\nmid e_n$.  Recall that
$L=K(x_1,\ldots,x_n)$, where $x_i\in K^{sep}$ satisfy
$x_1^p-x_1=\beta$, and $x_i^p-x_i=\omega_i^{p^n}\beta+\epsilon_i$ for
$2\le i\le n$.  For $0\le i\le n$ let $K_i=K(x_1,\ldots,x_i)$; then
$K=K_0$ and $L=K_n$.  Let $\sigma_1,\ldots,\sigma_n$ be generators for
$G=\Gal(L/K)$ such that $(\sigma_i-1)(x_j)=\delta_{ij}$.  As in the
proof of Proposition~\ref{LK1}, by applying Lemma~\ref{zE} to
$K(x_1,x_i)/K$ for $2\le i\le n$, we get $\z_i,E_i\in K_1$ such that
$\z_i^p-\z_i=E_i-\epsilon_i$ and $v_{K_1}(\z_i)=v_{K_1}(E_i)=-e_i$.
Set $\X_i=x_i-\omega_i^{p^{n-1}}x_1+\z_i$.  Then by Lemma~\ref{zE} we
get $\X_i^p-\X_i=-\wp(\omega_i)^{p^{n-1}}x_1+E_i$.  For $1\le i\le n$
we have $K_i=K_1(\X_2,\ldots,\X_i)$, and for $2\le i\le n$ we have
$-\omega_i^{p^{n-1}}x_1+\z_i\in K_1$.  It follows that
$(\sigma_j-1)(\X_i)=\delta_{ij}$ for $2\le i,j\le n$.

     For $2\le i\le n$ construct $\tilde{Y}_i$ for the
extension $K_i/K_1$ using the Artin-Schreier equations
\[\X_j^p-\X_j=-\wp(\omega_j)^{p^{n-1}}x_1+E_j\]
for $2\le j\le i$, just as $Y$ was constructed in
(\ref{Y}).  (Thus $n$ is replaced by $i-1$, $\beta$ is
replaced by $x_1$, and $\omega_j$ is replaced by
$-\wp(\omega_{j+1})^{p^{n-i}}$.)  Using (\ref{cofactor})
we write $\tilde{Y}_i=\tilde{t}_{i2}\X_2+\cdots
+\tilde{t}_{ii}\X_i$
with $\tilde{t}_{ij}\in\F$.  By Lemma~\ref{nonz} we
see that $\tilde{t}_{ij}\not=0$ for $2\le j\le i\le n$.
Therefore we may define $Y_i=\tilde{Y}_i/\tilde{t}_{ii}$.
Then $Y_i=t_{i2}\X_2+\cdots+t_{ii}\X_i$
with $t_{ij}=\tilde{t}_{ij}/\tilde{t}_{ii}\in
\F^{\times}$ and $t_{ii}=1$.  We also define
$Y_1=x_1$.

     For $2\le j\le i\le n$ we have
\begin{equation} \label{tij}
(\sigma_j-1)(Y_i)=\sum_{h=2}^i\,t_{ih}(\sigma_j-1)(\X_h)
=t_{ij}.
\end{equation}
Additionally, since
$\X_h=x_h-\omega_h^{p^{n-1}}x_1+\z_h$ we have
\[(\sigma_1-1)(Y_i)=\sum_{h=2}^i
t_{ih}(-\omega_h^{p^{n-1}}+(\sigma_1-1)(\z_h))\]
for $j=1\le i\le n$.
Hence $(\sigma_1-1)(Y_i)=t_{i1}+Z_i$, where
\begin{align*}
t_{i1}&=\dst-\sum_{h=2}^it_{ih}\omega_h^{p^{n-1}}
\in\F\\
Z_i&=\sum_{h=2}^it_{ih}(\sigma_1-1)(\z_h)\in K_1. 
\end{align*}
For $2\le i\le n-1$ we get
\begin{align*}
v_{K_1}(Z_i)
&\ge\min\{v_{K_1}((\sigma_1-1)(\z_h)):2\le h\le i\} \\
&\ge\min\{b-e_h:2\le h\le i\} \\
&> b-e_n.
\end{align*}

Recall that we have assumed $e_n\geq b/p^{n-1}>0$, and thus
$p\nmid e_n$. This means that
$v_{K_1}((\sigma_1-1)(\z_n))=b-e_n$.  Since
$t_{nn}\in\F^{\times}$ we get $v_{K_1}(Z_n)=b-e_n$.
Observe that
$v_{K_1}(Z_i)>0$ for $2\le i\le n$.
By (\ref{Y}) and elementary column operations we get
\begin{align*}
t_{i1}&=-\tilde{t}_{ii}^{-1}\det[\phi^{n-1}(\vomega^{(i)}),
\phi(-\phi^{n-i}(\wp(\vomega^{(i)}))),
\ldots,\phi^{i-2}(-\phi^{n-i}(\wp(\vomega^{(i)})))] \\
&=-\tilde{t}_{ii}^{-1}\det([\phi^{n-1}(\vomega^{(i)}),
\phi^{n-i+1}(\vomega^{(i)})-\phi^{n-i+2}(\vomega^{(i)}),
\ldots,\phi^{n-2}(\vomega^{(i)})-\phi^{n-1}(\vomega^{(i)})])
\\
&=-\tilde{t}_{ii}^{-1}\det([\phi^{n-1}(\vomega^{(i)}),
\phi^{n-i+1}(\vomega^{(i)}),\phi^{n-i+2}(\vomega^{(i)}),
\ldots,\phi^{n-2}(\vomega^{(i)})])
\end{align*}
where
\[\vomega^{(i)}=\begin{bmatrix}\omega_2\\
\vdots\\ \omega_i\end{bmatrix}.\]
Hence by Lemma~\ref{nonz} we have $t_{i1}\not=0$.

We are now prepared to construct $\Psi_j$ for $1\leq j\leq n$.
Following \cite[Definition 2.7]{large}, we 
define $\Theta_n,\Theta_{n-1},\ldots,\Theta_1$
iteratively by $\Theta_n=\sigma_n$ and
\[\Theta_j=\sigma_j\Theta_n^{[-t_{nj}]}
\Theta_{n-1}^{[-t_{n-1,j}]}\ldots \Theta_{j+1}^{[-t_{j+1,j}]}.\] Then
$\Theta_j\in \F[\sigma_j,\sigma_{j+1},\ldots,\sigma_n]$.  Set
$\Psi_j=\Theta_j-1$.  It remains to prove that these $\Psi_j\in \F[G]$
have the desired properties.

First we consider $\Psi_j$ for $2\leq j\leq n$.
The
scaffold for $L/K_1$ given by Proposition~\ref{scafr} has
the form $(\{\lambda_w\},\{\Psi_j\}_{2\le j\le n})$ for
some $\lambda_w\in L$.  Since $\rho$ is a valuation
criterion element for $L/K$ we have
$v_L(\rho)\equiv b\pmod{p^{n-1}}$, so $\rho$ is also a
valuation criterion element for $L/K_1$.  Since
$\Psi_j=\Psi^{(p^{n-j})}$, by equation (\ref{Psit}) we
have $v_L(\Psi_j(\rho))-v_L(\rho)=bp^{n-j}$ for
$2\le j\le n$. These $\Psi_j$ have the needed properties.

Now we consider $\Psi_1$. In checking the needed properties we may
work with any valuation criterion element $\rho\in L$. So choose
\begin{equation} \label{binom}
\rho=\binom{Y_1}{p-1}\binom{Y_2}{p-1}
\ldots\binom{Y_n}{p-1}.
\end{equation}
Indeed, since $v_L(Y_h)=-bp^{n-h}<0$, for $0\le r\le p-1$ we have
$ v_L\left(\binom{Y_h}{r}\right)=-rbp^{n-h}$.  Hence $\rho$
satisfies $v_L(\rho)=-(p^n-1)b\equiv b\pmod{p^n}$, so
$\rho$ is a valuation criterion element for $L/K$.

To compute $v_L(\Psi_1(\rho))-v_L(\rho)$, we will need certain details
from \cite{large}.  Let $0\le s<p^n$ and write
\[s=s_n+s_{n-1}p+\cdots+s_1p^{n-1}\]
with $0\le s_i<p$.  Define
\[\binom{Y}{s}=\binom{Y_n}{s_n}\binom{Y_{n-1}}{s_{n-1}}
\ldots\binom{Y_1}{s_1}.\]
Since $v_L(Y_i)=p^{n-i}v_{K_i}(Y_i)=-bp^{n-i}<0$ we have
$v_L\left(\binom{Y}{s}\right)=-bs$.

\begin{prop} \label{psijYs}
For $2\le j\le n$ we have
\[\Psi_j\left(\binom{Y}{s}\right)
=\binom{Y_n}{s_n}\ldots\binom{Y_j}{s_j-1}\ldots
\binom{Y_1}{s_1}.\]
In particular, if $s_j=0$ then
$\dst\Psi_j\left(\binom{Y}{s}\right)=0$.
\end{prop}

\begin{proof}
This follows from \cite[Proposition 2.13]{large}.  We
include the proof, since it leads naturally to
Lemma~\ref{alpha}, which is needed to handle the case
$j=1$.  Use reverse induction on $j$.  Since
\[\sigma_n\left(\binom{Y_n}{s_n}\right)
=\binom{Y_n+1}{s_n}=\binom{Y_n}{s_n}+\binom{Y_n}{s_n-1}\]
we get $\dst\Psi_n\left(\binom{Y_n}{s_n}\right)=
\binom{Y_n}{s_n-1}$ and hence
\[\Psi_n\left(\binom{Y}{s}\right)
=\binom{Y_n}{s_n-1}\binom{Y_{n-1}}{s_{n-1}}
\ldots\binom{Y_1}{s_1}.\]
Let $1\le j<n$ and assume the claim holds for $j+1$.
Then
\begin{align*}
\Psi_j\left(\binom{Y_n}{s_n}\ldots
\binom{Y_j}{s_j}\right)
&=\sigma_j\Theta_n^{[-t_{nj}]}\ldots
\Theta_{j+1}^{[-t_{j+1,j}]}
\left(\binom{Y_n}{s_n}\ldots\binom{Y_j}{s_j}\right)
-\binom{Y_n}{s_n}\ldots\binom{Y_j}{s_j}.
\end{align*}
To make further progress we need a lemma.

\begin{lemma} \label{alpha}
Let $1\le j\le n-1$ and assume that the inductive
hypothesis holds for $j+1$.  Let $h$ satisfy
$j+1\le h\le n$ and let $\alpha\in\F$.  Then
\[\Theta_h^{[\alpha]}\left(\binom{Y_n}{s_n}\ldots
\binom{Y_j}{s_j}\right)
=\binom{Y_n}{s_n}\ldots\binom{Y_h+\alpha}{s_h}\ldots
\binom{Y_j}{s_j}.\]
\end{lemma}

\begin{proof}
Using the inductive hypothesis for $h$ we get
\begin{align*}
\Theta_h^{[\alpha]}\left(\binom{Y_n}{s_n}\ldots
\binom{Y_j}{s_j}\right)
&=\sum_{r=0}^{p-1}\binom{\alpha}{r}\Psi_h^r
\left(\binom{Y_n}{s_n}\ldots\binom{Y_j}{s_j}\right) \\
&=\sum_{r=0}^{s_h}\binom{\alpha}{r}
\binom{Y_n}{s_n}\ldots\binom{Y_h}{s_h-r}\ldots
\binom{Y_j}{s_j} \\
&=\binom{Y_n}{s_n}\ldots\binom{Y_h+\alpha}{s_h}\ldots
\binom{Y_j}{s_j},
\end{align*}
where the last equality follows from the Vandermonde
convolution identity.
\end{proof}

     It follows from the lemma and equation (\ref{tij})
that
\begin{align*}
\Theta_j\left(\binom{Y_n}{s_n}\ldots
\binom{Y_j}{s_j}\right)
&=\sigma_j\left(\binom{Y_n-t_{nj}}{s_n}\ldots
\binom{Y_{j+1}-t_{j+1,j}}{s_{j+1}}
\binom{Y_j}{s_j}\right) \\
&=\sigma_j\left(\binom{Y_n-t_{nj}}{s_n}\right)\ldots
\sigma_j\left(\binom{Y_{j+1}-t_{j+1,j}}{s_{j+1}}\right)
\sigma_j\left(\binom{Y_j}{s_j}\right) \\
&=\binom{Y_n}{s_n}\ldots\binom{Y_{j+1}}{s_{j+1}}
\binom{Y_j+1}{s_j}.
\end{align*}
Therefore we have
\begin{align*}
\Psi_j\left(\binom{Y_n}{s_n}\ldots
\binom{Y_j}{s_j}\right)
&=\binom{Y_n}{s_n}\ldots\binom{Y_{j+1}}{s_{j+1}}
\binom{Y_j+1}{s_j}-\binom{Y_n}{s_n}\ldots
\binom{Y_{j+1}}{s_{j+1}}\binom{Y_j}{s_j} \\
&=\binom{Y_n}{s_n}\ldots\binom{Y_{j+1}}{s_{j+1}}
\binom{Y_j}{s_j-1}
\end{align*}
It follows that
\[\Psi_j\left(\binom{Y}{s}\right)
=\binom{Y_n}{s_n}\ldots\binom{Y_j}{s_j-1}\ldots
\binom{Y_1}{s_1}.\]
This completes the proof of Proposition~\ref{psijYs}.
\end{proof}

     We now fill in the missing case of
Proposition~\ref{psijYs} by computing
$\Psi_1(\binom{Y}{s})$.  We focus on the case $s=p^n-1$
since $\Psi_1(\rho)=\Psi_1(\binom{Y}{p^n-1})$.  By
Lemma~\ref{alpha} we have
\begin{align}
\Theta_1\left(\binom{Y}{s}\right)
&=\sigma_1\Theta_n^{[-t_{n1}]}\ldots
\Theta_2^{[-t_{21}]}
\left(\binom{Y_n}{s_n}\ldots\binom{Y_1}{s_1}\right)
\nonumber \\
&=\sigma_1\left(\binom{Y_n-t_{n1}}{s_n}\ldots
\binom{Y_2-t_{n1}}{s_2}\binom{Y_1}{s_1}\right)
\nonumber \\
&=\binom{Y_n+Z_n}{s_n}\ldots\binom{Y_2+Z_2}{s_2}
\binom{Y_1+1}{s_1} \nonumber \\
&=\prod_{h=1}^n\binom{Y_h+Z_h}{s_h}, \label{YhZh}
\end{align}
where we let $Z_1=1$ for notational convenience.  It
follows from Vandermonde's convolution identity that for
$1\le h\le n$ we have
\begin{align} \nonumber
\binom{Y_h+Z_h}{p-1}
&=\sum_{r=0}^{p-1}\binom{Y_h}{p-1-r}\binom{Z_h}{r} \\
&=\binom{Y_h}{p-1}+\binom{Y_h}{p-2}\binom{Z_h}{1}+\cdots.
\label{omit}
\end{align}
Since $v_L(Y_h)<0\le v_L(Z_h)$, the terms which are
omitted from (\ref{omit}) all have larger valuation than
the two terms which are written explicitly.  It follows
from (\ref{YhZh}) that
\begin{align} \nonumber
\Theta_1\left(\binom{Y}{p^n-1}\right)
&=\prod_{h=1}^n\left(\binom{Y_h}{p-1}
+\binom{Y_h}{p-2}\binom{Z_h}{1}+\cdots\right) \\
&=\prod_{h=1}^n\binom{Y_h}{p-1}
+\sum_{h=1}^n\binom{Y_h}{p-2}\binom{Z_h}{1}
\prod_{g\not=h}\binom{Y_g}{p-1}+\cdots \nonumber \\
&=\binom{Y}{p^n-1}+\sum_{h=1}^n\binom{Y}{p^n-p^{n-h}-1}
\binom{Z_h}{1}+\cdots
\nonumber \\[.2cm]
\Psi_1\left(\binom{Y}{p^n-1}\right)\
&=\Theta_1\left(\binom{Y}{p^n-1}\right)-\binom{Y}{p^n-1}
\nonumber \\
&=\sum_{h=1}^n\binom{Y}{p^n-p^{n-h}-1}\binom{Z_h}{1}
+\cdots \label{right}
\end{align}
We claim that the valuation of
$\Psi_1(\binom{Y}{p^n-1})$ is  the minimum of the valuations of the $h=1$ and $h=n$ summands of (\ref{right}).
The valuation of the
$h$th summand is
\begin{align*}
v_L\left(\binom{Y}{p^n-p^{n-h}-1}\binom{Z_h}{1}\right)
&=v_L(\rho)-v_L(Y_h)+v_L(Z_h) \\
&=v_L(\rho)+bp^{n-h}+v_L(Z_h).
\end{align*}
For $2\le h\le n-1$ we have $v_{K_1}(Z_h)>
v_{K_1}(Z_n)$, and hence
\[bp^{n-h}+v_L(Z_h)>b+v_L(Z_n).\]
Since $p\nmid b$ and $v_L(Z_1)=0$ we have
\begin{align*}
b+v_L(Z_n)=b+(b-e_n)p^{n-1}\not=bp^{n-1}
=bp^{n-1}+v_L(Z_1).
\end{align*}
This verifies the claim.
It follows that
\begin{align} \nonumber
\hat{v}_{\rho}(\Psi_1)&=v_L(\Psi_1(\rho))-v_L(\rho) \\
&=\min\{bp^{n-1},(b-e_n)p^{n-1}+b\}. \label{Psi1}
\end{align}

\subsubsection{Proof that $\hat{v}_{\rho}(\Psi_i')\le \hat{v}_{\rho}(\Psi_i)$ for
  all $\Psi_i'\in\Psi_i+A^2$}

Assume for a contradiction that there are $1\leq i\leq n$ and
$\Psi_i'\in\Psi_i+A^2$ such that $v_L(\Psi_i'(\rho))>
v_L(\Psi_i(\rho))$.  Then there exists an element $\Upsilon\in A^2$,
namely $\Upsilon=\Psi_i-\Psi_i'$, such that
$v_L(\Upsilon(\rho))=v_L(\Psi_i(\rho))$.  Based upon the recursive
definition of the $\Theta_j$, the $\Psi_j$ generate $A$. Thus, we can
express $\Upsilon$ as a polynomial in $\Psi_1,\ldots,\Psi_n$ in which
all terms have degree at least 2.  In other words,
\[\Upsilon=\sum_{s=1}^{p^n-1}a_s\Psi^{(s)}\]
with $a_s\in\F$ for all $s$ and $a_s=0$ if $s$ is a power of $p$.
Recall that for $s=s_{(0)}+s_{(1)}p+\cdots+s_{(n-1)}p^{n-1}$ with
$0\le s_{(j)}\le p-1$, we have
$\Psi^{(s)}=\Psi_n^{s_{(0)}}\Psi_{n-1}^{s_{(1)}}\cdots
\Psi_1^{s_{(n-1)}}$.
If $p^{n-1}\le s<p^n$ then
$s_{(n-1)}\not=0$, and if $a_s\neq  0$, then $s\neq p^{n-1}$. Hence
for such $s$,
\[v_L(\Psi^{(s)}(\rho))>v_L(\Psi_1(\rho))\ge
v_L(\Psi_i(\rho)).\]
It follows that 
\[\Upsilon_0=\sum_{s=1}^{p^{n-1}-1}a_s\Psi^{(s)}\]
satisfies $v_L(\Upsilon_0(\rho))
=v_L(\Upsilon(\rho))=v_L(\Psi_i(\rho))$.

Since there is a scaffold for $L/K_1$ of the form
$(\{\lambda_w\},\{\Psi_j\}_{2\le j\le n})$, it follows from equation
(\ref{Psit}) that the valuations of the nonzero terms of
$\Upsilon_0(\rho)$ are distinct.  Hence there is $1\le s<p^{n-1}$ such
that $a_s\not=0$ and
$v_L(\Psi^{(s)}(\rho))=v_L(\Upsilon_0(\rho))=v_L(\Upsilon(\rho))=v_L(\Psi_i(\rho))$.

We consider the cases $2\le i\le n$ and $i=1$ separately.  If $2\le
i\le n$ then because $\Psi_i=\Psi^{(p^{n-i})}$, we have $s=p^{n-i}$,
which implies $a_s=0$, a contradiction.  If $i=1$, then by equations
\eqref{Psit} and \eqref{Psi1}
we get
\begin{align} \nonumber
v_L(\Psi^{(s)}(\rho))-v_L(\rho)
&=v_L(\Psi_1(\rho))-v_L(\rho) \\
sb&=\min\{bp^{n-1},(b-e_n)p^{n-1}+b\}. \label{sb}
\end{align}
Since $s<p^{n-1}$ we have
$sb\not=bp^{n-1}$.  If $sb=(b-e_n)p^{n-1}+b$ then
$sb\equiv b\pmod{p^{n-1}}$, and hence $s=1$.  Since
$a_1=0$ this is a contradiction.

We therefore conclude that
 $v_L(\Psi_i'(\rho))\le v_L(\Psi_i(\rho))$ for all
$\Psi_i'\in\Psi_i+A^2$, and thus that 
\begin{align*}
\hat{v}_{\rho}(\Psi_i+A^2)&=\hat{v}_{\rho}(\Psi_i)
=\begin{cases}
\min\{bp^{n-1},(b-e_n)p^{n-1}+b\}&(i=1) \\
bp^{n-i}&(2\le i\le n).
\end{cases}
\end{align*}
Using Corollary~\ref{same} we deduce that the set of
VC$_k$-refined breaks of $L/K$ is defined and equal to
$B_1$, and the set of SS$_k$-refined breaks of $L/K$ is
equal to $B_1$.  This completes the proof of
Theorem~\ref{main}.~\medskip

     When we specialize Theorem~\ref{main} to the case
$n=2$ the hypotheses on the $e_i$ reduce to $e_2<b$,
which holds by the definition of Artin-Schreier data.
Therefore we have the following  characteristic-$p$ analog of
\cite[Theorem 5]{necessity}.

\begin{cor}
Let $L/K$ be a totally ramified $C_p^2$-extension with
a single ramification break $b>0$ and let
$(\beta,\vomega,\vepsilon)$ be Artin-Schreier data for
$L/K$ such that $\omega_1=1$ and $\epsilon_1=0$.  Set
$b_*=\min\{(b-e_2)p+b,bp\}$ and $B_1=\{b,b_*\}$.  Then for
$2\le k\le p$ the set of VC$_k$-refined breaks of $L/K$
and the set of SS$_k$-refined breaks of $L/K$ are both
equal to $B_1$.
\end{cor}

\section{Concluding remarks} \label{concluding}

We finish with two topics. Firstly, we discuss how our refined
breaks relate to Galois module theory and to other generalizations of
ramification data. Secondly, we discuss the class of extensions in
section~\ref{class} for which,
based upon
Corollary~\ref{same},
the VC$_2$-refined breaks are defined and
equivalent to the SS$_2$-refined breaks.

Firstly, let $2\le k\le p$, and observe that for $h\ge1$ we have
\[\{\gamma\in A:\hat{v}_L(\gamma)\ge h\}=
\bigcap_{r=0}^{p^n-1}\Ann_R(\M_L^r/\M_L^{r+h}),\]
where $\Ann_R(M)$ refers to the annihilator in $R$ of
the $R$-module $M$.  Therefore for
$\gammabar\in\Gbar^{[\F]}$ we have
$\hat{v}_L(\gammabar-\onebar)\ge h$ if and only if there
is $\gamma'\in\gammabar$ such that $\gamma'-1$ lies in
the intersection of the annihilators of
$\M_L^r/\M_L^{r+h}$ for $0\le r\le p^n-1$.  It follows
that the SS$_k$-refined breaks of $L/K$ can be computed
in terms of the Galois module structure of quotients of
$\OO_L$-ideals.  Hence any set of invariants which
completely determines the $\OO_K[G]$-module structures
of quotients of ideals in $\OO_L$ must also determine
the SS$_k$-refined breaks.  The same holds for the
$(\rho,k)$-refined breaks for any fixed valuation
criterion element $\rho$, and also for the
VC$_k$-refined breaks when they are defined.  

Restrict now to $k=2$ and suppose that the VC$_2$-refined breaks of
$L/K$ are defined and equal to the SS$_2$-refined breaks; recall the
sufficient conditions for this in Corollary~\ref{same}.  In this case
there is a tighter interpretation of the refined breaks in terms of
Galois module theory: We have $\onebar+\deltabar\in\Gbar_h^{[\F]}$ if
and only if $\delta\in\Ann_R(\M_L^b/\M_L^{b+h})+A^2$.  Hence $h$ is a
refined break of $L/K$ if and only if
\[\dim_{\F}((\Ann_R(\M_L^b/\M_L^{b+h+1})+A^2)/A^2)<
\dim_{\F}((\Ann_R(\M_L^b/\M_L^{b+h})+A^2)/A^2).\]

     We contrast these results with two others.  In
\cite{necessity} it is shown that the Galois module
structure of $\OO_L$-ideals (rather than quotients of
ideals) determines the refined breaks in some cases.
We don't know whether the Galois module structure of
$\OO_L$-ideals is enough to determine our breaks.  On
the other hand, it is not obvious that the indices of
inseparability of $L/K$ can be determined from any sort
of Galois module structure, even though the indices of
inseparability determine the refined breaks in some
cases \cite{refined}.

Now we discuss the extensions in Corollary~\ref{same}.
In the introduction to
\cite{scaf} extensions with a Galois scaffold, such as
those in section~\ref{class}, are said to be, in a
certain Galois module theory sense, as simple as
ramified cyclic extensions of degree $p$.  Indeed,
this assertion motivated their construction, an
 assertion that is now justified by
\cite{bce} where 
Galois module structure results from \cite{lara} that
were only known for cyclic extensions of degree $p$ have been generalized to
all extensions with a Galois scaffold of sufficiently
high precision.  In section~\ref{class} we introduce a
family of totally ramified $C_p^n$-extensions that
includes all totally ramified $C_p^2$-extensions with
one ramification break.  Based upon Corollary~\ref{same}, each
extension in this family can now, from another Galois module
theory perspective, be said to be as simple as a
totally ramified $C_p^2$-extension with one break.


\begin{thebibliography}{99}

\bibitem{bon2} M. V. Bondarko, Local Leopoldt's problem
for ideals in totally ramified $p$-extensions of
complete discrete valuation fields, Algebraic number
theory and algebraic geometry, 27--57, Contemp.\ Math.\
{\bf300}, Amer.\ Math.\ Soc.\ Providence, RI, 2002.
  
\bibitem{new} N. P. Byott and G. G. Elder, New
ramification breaks and additive Galois structure, J.
Th\'eor.\ Nombres Bordeaux {\bf 17} (2005), 87--107.

\bibitem{proc-lms} N. P. Byott and G. G. Elder, A valuation criterion
  for normal bases in elementary abelian extensions.
  Bull.\ Lond.\ Math.\ Soc.\ {\bf 39} (2007), 705--708.

\bibitem{necessity} N. P. Byott and G. G. Elder, On the
necessity of new ramification breaks, J. Number Theory
{\bf 129} (2009), 84--101.

\bibitem{large} N. P. Byott and G. G. Elder, Sufficient
  conditions for large Galois scaffolds, J. Number Theory
  {\bf 182} (2018), 95--130.

  
\bibitem{bce} N. P. Byott, L. N. Childs, and G. G.
Elder, Scaffolds and Generalized Integral Galois Module
Structure, Ann.\ Inst.\ Fourier (Grenoble), {\bf 68} (2018), 965--1010.

\bibitem{lara} B. de Smit, and L. Thomas, Local Galois module structure in
  positive characteristic and continued fractions, Arch. Math. (Basel)
  {\bf 88} (2007) 207--219.

\bibitem{scaf} G. G. Elder, Galois scaffolding in
one-dimensional elementary abelian extensions, Proc.\
Amer.\ Math.\ Soc.\ {\bf137} (2009), 1193--1203.

\bibitem{valp} G. G. Elder, A valuation criterion for
normal basis generators in local fields of
characteristic $p$, Arch.\ Math.\ {\bf94} (2010),
43--47.

\bibitem{FV} I. B. Fesenko and S. V. Vostokov, {\em
Local fields and their extensions}, Amer.\ Math.\ Soc.,
Providence, RI, 2002.

\bibitem{fried} M. Fried, Arithmetical properties of
function fields II, The generalized Schur problem, Acta
Arith.\ {\bf25} (1973/74), 225--258.

\bibitem{goss} D. Goss, {\em Basic structures of function field
  arithmetic}, Springer-Verlag, Berlin, 1996.


\bibitem{heier} V. Heiermann, De nouveaux invariants
num\'eriques pour les extensions totalement ramifi\'ees
de corps locaux, J. Number Theory {\bf 59} (1996),
159--202.

\bibitem{refined} K. Keating, Indices of inseparability
and refined ramification breaks, J. Number Theory
{\bf 142} (2014), 1--17.

\bibitem{elem} K. Keating, Indices of inseparability for
elementary abelian $p$-extensions, J. Number Theory
{\bf136} (2014), 233--251.

\bibitem{rees} D. Rees, Valuations associated with a
local ring I, Proc.\ London Math.\ Soc.\ {\bf3} (1955),
107--128.

\bibitem{cl} J.-P. Serre, {\em Corps Locaux}, Hermann,
Paris, 1962.

\end{thebibliography}
\end{document}